\documentclass[a4paper,11pt,fleqn]{article}
\usepackage{amsmath}
\usepackage{amssymb}
\usepackage{theorem}
\usepackage{mathrsfs} 
\usepackage[usenames,dvipsnames]{pstricks}
\usepackage{pstricks-add}
\usepackage{pst-plot}
\usepackage{pst-math}
\usepackage{euscript}
\usepackage{graphicx}
\usepackage{enumitem}
\usepackage{charter}
\usepackage{empheq}
\usepackage{srcltx}

\newcommand{\email}[1]{\href{mailto:#1}{\nolinkurl{#1}}}
\definecolor{labelkey}{rgb}{0,0.08,0.45}
\definecolor{refkey}{rgb}{0,0.6,0.0}
\definecolor{Brown}{rgb}{0.45,0.0,0.05}
\definecolor{dgreen}{rgb}{0.00,0.59,0.00}
\definecolor{dblue}{rgb}{0,0.08,0.75}
\RequirePackage[dvips,colorlinks,hyperindex]{hyperref} 
\hypersetup{linktocpage=true,citecolor=dblue,linkcolor=dgreen}
\PassOptionsToPackage{normalem}{ulem}
\usepackage{ulem}

\renewcommand{\leq}{\ensuremath{\leqslant}}
\renewcommand{\geq}{\ensuremath{\geqslant}}

\newcommand{\Frac}[2]{\displaystyle{\frac{#1}{#2}}} 
 
\newcommand{\scal}[2]{{\left\langle{{#1}\mid{#2}}\right\rangle}}
\newcommand{\abscal}[2]{\left|\left\langle{{#1}\mid{#2}}%
\right\rangle\right|} 

\newcommand{\menge}[2]{\big\{{#1}~\big |~{#2}\big\}} 
\newcommand{\Menge}[2]{\left\{{#1}~\Big|~{#2}\right\}}

\newcommand{\HH}{\ensuremath{{\mathcal H}}}

\newcommand{\GG}{\ensuremath{{\mathcal G}}}

\newcommand{\XX}{\ensuremath{\mathcal{X}}}

\newcommand{\emp}{\ensuremath{{\varnothing}}}

\newcommand{\Id}{\ensuremath{\operatorname{Id}}}

\newcommand{\RR}{\ensuremath{\mathbb{R}}}
\newcommand{\RP}{\ensuremath{\left[0,+\infty\right[}}

\newcommand{\RPP}{\ensuremath{\left]0,+\infty\right[}}

\newcommand{\RX}{\ensuremath{\left]-\infty,+\infty\right]}}

\newcommand{\NN}{\ensuremath{\mathbb N}}
\newcommand{\JJ}{\ensuremath{\mathbb J}}

\newcommand{\bK}{\ensuremath{\mathbb K}}

\newcommand{\exi}{\ensuremath{\exists\,}}
\newcommand{\ran}{\ensuremath{\text{\rm ran}\,}}

\newcommand{\pinf}{\ensuremath{{+\infty}}}

\newcommand{\dom}{\ensuremath{\text{\rm dom}\,}}
\newcommand{\proj}{\ensuremath{\text{\rm proj}}}
\newcommand{\prox}{\ensuremath{\text{\rm prox}}}

\newcommand{\gra}{\ensuremath{\text{\rm gra}\,}}

\newcommand{\conv}{\ensuremath{\text{\rm conv}\,}}

\newcommand{\rzeroun}{\ensuremath{\left]0,1\right]}} 
 
\newcommand{\czeroun}{\ensuremath{\left[0,1\right]}} 
\definecolor{dgreen}{rgb}{0.00,0.59,0.00}

\newtheorem{theorem}{Theorem}[section]
\newtheorem{lemma}[theorem]{Lemma}
\newtheorem{corollary}[theorem]{Corollary}
\newtheorem{proposition}[theorem]{Proposition}
\theoremstyle{plain}{\theorembodyfont{\rmfamily}%
\newtheorem{model}[theorem]{Model}}
\theoremstyle{plain}{\theorembodyfont{\rmfamily}%
\newtheorem{notation}[theorem]{Notation}}
\theoremstyle{plain}{\theorembodyfont{\rmfamily}%
}
\theoremstyle{plain}{\theorembodyfont{\rmfamily}%
\newtheorem{assumption}[theorem]{Assumption}}
\theoremstyle{plain}{\theorembodyfont{\rmfamily}%
}
\theoremstyle{plain}{\theorembodyfont{\rmfamily}%
}
\theoremstyle{plain}{\theorembodyfont{\rmfamily}%
\newtheorem{example}[theorem]{Example}}
\theoremstyle{plain}{\theorembodyfont{\rmfamily}%
\newtheorem{remark}[theorem]{Remark}}
\theoremstyle{plain}{\theorembodyfont{\rmfamily}%
}
\theoremstyle{plain}{\theorembodyfont{\rmfamily}%
}

\numberwithin{equation}{section}

\begin{document}
\title{\sffamily \vskip -9mm 
Lipschitz Certificates for Layered Network Structures Driven 
by Averaged Activation Operators\thanks{Contact author: 
P. L. Combettes, \email{plc@math.ncsu.edu}, phone: +1 919 515 2671.
The work of P. L. Combettes was supported by the 
National Science Foundation under grant CCF-1715671. 
The work of J.-C. Pesquet was supported by Institut 
Universitaire de France.}}
\author{Patrick L. Combettes$^1$ and Jean-Christophe Pesquet$^2$
\\[4mm]
\small
\small $\!^1$North Carolina State University,
Department of Mathematics, Raleigh, NC 27695-8205, USA\\
\small\email{plc@math.ncsu.edu}\\[2mm]
\small $\!^2$CentraleSup\'elec, Inria, Universit\'e Paris-Saclay,
Center for Visual Computing, 91190 Gif sur Yvette, France\\
\small\email{jean-christophe@pesquet.eu}
}
\date{~}
\maketitle

\begin{abstract}
Obtaining sharp Lipschitz constants for feed-forward neural networks
is essential to assess their robustness in the face of
perturbations of their inputs. We derive such constants in the
context of a general layered network model involving compositions
of nonexpansive averaged operators and affine operators. 
By exploiting this architecture, our analysis finely
captures the interactions between the layers, yielding tighter
Lipschitz constants than those resulting from the product of
individual bounds for groups of layers. 
The proposed framework
is shown to cover in particular many practical instances
encountered in feed-forward neural networks.
Our Lipschitz constant estimates are further
improved in the case of structures employing scalar nonlinear
functions, which include standard convolutional networks as special
cases. 
\end{abstract}


\section{Introduction}

Artificial neural networks are becoming increasingly central tools
in tasks such as learning, modeling, data processing, and decision
making. As first noted in \cite{Szeg13}, neural networks are
vulnerable to adversarial examples which, though close to other 
data inputs, lead to very different outputs. This potential lack of
stability makes the networks vulnerable and unreliable in key
application areas; see, for instance, 
\cite{Akta18,Good15,Kreu18} and the references therein. 
To protect networks against such instabilities various
techniques have been explored \cite{Madr18,Pape16,Ragh18,Wong18}. 
Although these defense strategies may be effective 
in certain scenarios, they do not provide formal guarantees of 
robustness for general networks and they have been shown to be
breakable by new attacks; see, for instance,
\cite{Atha18,Car17b}.

It has been acknowledged for some time that the Lipschitz behavior
of a network plays a key role in the analysis of its robustness
\cite{Szeg13}. Simply put, if a layered network is modeled by
an operator $T$ acting between normed spaces, with Lipschitz
constant $\theta$, given an input $x$ and a perturbation $z$, we 
can majorize the perturbation on the output via the inequality
\begin{equation}
\|T(x+z)-Tx\|\leq\theta\|z\|.
\end{equation}
Thus $\theta$ can be used as a certificate of robustness of the
network provided that it is tightly estimated. 
Lipschitz regularity is also an important ingredient 
in the derivation of generalization bounds and approximation bounds 
\cite{Bart17,Bolc19,Soko17}, and
of reachability conditions \cite{Ruan18}.
In \cite{Szeg13} the estimation of $\theta$ is
performed by evaluating the Lipschitz constant of the layers
individually and then defining $\theta$ as the product of these
constants, which typically yields pessimistic bounds. Lipschitz
constants have also been computed for specific situations, e.g.,
\cite{Bala17,Hein17,Scam18,Tsuz18}. Overall, however,
deriving analytically accurate constants for
general contexts remains an open problem. The objective of
the present paper is to address this question for 
a general class of layered networks. 
Mathematically, our network model is described as an
alternation of affine and nonlinear operators. 
This type of structure also arises in variational and equilibrium
problems, as well as in network science, e.g., 
\cite{Bric13,2019,Facc03,Yipa19}.
Adopting the same terminology as in the neural network literature, 
where they model the activity of neurons,
the nonlinear operators will be called activation operators.
Our stability analysis 
focuses on the following $m$-layer model, in which 
the activation operators are averaged nonexpansive operators (see
Fig.~\ref{fig:1}). Recall that an operator $R\colon\HH\to\HH$
acting on a Hilbert space $\HH$ is $\alpha$-averaged for some
$\alpha\in[0,1]$ if there exists a nonexpansive (i.e.,
$1$-Lipschitzian) operator $Q\colon\HH\to\HH$ such that
\begin{equation}
\label{e:averaged}
R=(1-\alpha)\Id+\alpha Q.
\end{equation}
In other words, $R=\Id+\alpha(Q-\Id)$ is an underrelaxation of 
a nonexpansive operator (see \cite{Livre1} for a detailed account).
This class of operators was introduced in \cite{Bail78} 
and shown in \cite{Opti04} to model various problems in nonlinear 
analysis as it includes common operators such as projection
operators, proximity operators, resolvents of monotone operators, 
reflection operators, gradient step operators, and various
combinations thereof. Recent theoretical developments and
applications to data science include 
\cite{Bayr16,Bert18,Borw17,Botr17,Brav18,Siop17,%
Cond13,Koyu19,Mour18,Suny19,Yama17,Yipa19}.

\begin{figure}
\label{fig:1}
\scalebox{0.75} 
{
\begin{pspicture}(-0.75,-1.7)(15.9,2.1)
\psline[linewidth=0.04cm,arrowsize=2.2mm]{->}(0.35,0.0)(1.0,0.0)
\psline[linewidth=0.04cm,arrowsize=2.2mm]{->}(3.0,0.0)(3.65,0.0)
\psline[linewidth=0.04cm,arrowsize=2.2mm]{->}(4.35,0.0)(5.0,0.0)
\psline[linewidth=0.04cm,arrowsize=2.2mm]{->}(4.0,1.2)(4.0,0.36)
\psframe[linewidth=0.04,dimen=outer](1.0,-1.0)(3.0,1.0)
\pscircle[linewidth=0.04,dimen=outer](4.0,0.0){0.35}
\rput(0.15,0.0){\large$\boldsymbol{{x}}$}
\rput(2.0,0.0){\large$\boldsymbol{{W_1}}$}
\rput(4.0,0.0){\large$\boldsymbol{+}$}
\rput(4.0,1.5){\large$\boldsymbol{{b_1}}$}
\rput(6.0,0.0){\large$\boldsymbol{{R_1}}$}
\psframe[linewidth=0.04,dimen=outer](5.0,-1.0)(7.0,1.0)
\psline[linewidth=0.04cm,arrowsize=2.2mm]{->}(7.00,0.0)(7.65,0.0)
\rput(8.5,0.0){\large$\boldsymbol{{\cdots}}$}
\psline[linewidth=0.04cm,arrowsize=2.2mm]{->}(9.35,0.0)(10.0,0.0)
\psline[linewidth=0.04cm,arrowsize=2.2mm]{->}(12.0,0.0)(12.65,0.0)
\psline[linewidth=0.04cm,arrowsize=2.2mm]{->}(13.35,0.0)(14.0,0.0)
\psline[linewidth=0.04cm,arrowsize=2.2mm]{->}(13.0,1.2)(13.0,0.36)
\psframe[linewidth=0.04,dimen=outer](10.0,-1.0)(12.0,1.0)
\pscircle[linewidth=0.04,dimen=outer](13.0,0.0){0.35}
\rput(11.0,0.0){\large$\boldsymbol{{W_m}}$}
\rput(13.0,0.0){\large$\boldsymbol{+}$}
\rput(13.0,1.5){\large$\boldsymbol{{b_m}}$}
\rput(15.0,0.0){\large$\boldsymbol{{R_m}}$}
\rput(17.0,0.05){\large$\boldsymbol{{Tx}}$}
\psframe[linewidth=0.04,dimen=outer](14.0,-1.0)(16.0,1.0)
\psline[linewidth=0.04cm,arrowsize=2.2mm]{->}(16.00,0.0)(16.65,0.0)
\end{pspicture} 
}
\vskip -4mm
\caption{In Model~\ref{m:1}, the $i$th layer involves a linear
weight operator $W_i$, a bias vector $b_i$, and an 
activation operator $R_i$, which is assumed to be a nonlinear 
averaged nonexpansive operator.}
\end{figure}
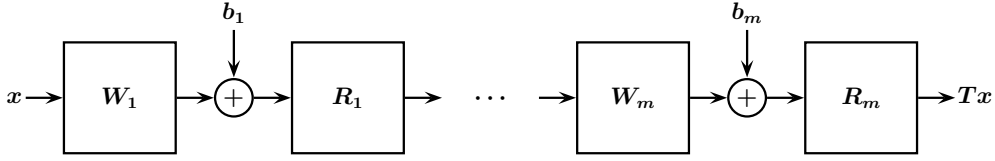

\begin{model}
\label{m:1}
Let $m\geq 1$ be an integer and let $(\HH_i)_{0\leq i\leq m}$ be 
nonzero real Hilbert spaces. For every $i\in\{1,\ldots,m\}$, 
let $W_i\colon\HH_{i-1}\to\HH_i$ be a bounded linear 
operator, let $b_i\in\HH_i$, let $\alpha_i\in [0,1]$, and let 
$R_i\colon\HH_i\to\HH_i$ be an $\alpha_i$-averaged operator. Set
\begin{equation}
\label{e:defTi}
T=T_m\circ\cdots\circ T_1,
\quad\text{where}\quad(\forall i\in\{1,\ldots,m\})\quad 
T_i\colon\HH_{i-1}\to\HH_i\colon x\mapsto R_i(W_ix+b_i).
\end{equation}
\end{model}

Since the operators $(R_i)_{1\leq i\leq m}$ are
nonexpansive, a Lipschitz constant for $T$ in 
\eqref{e:defTi} is 
\begin{equation}
\label{e:trivubound}
\theta_m=\prod_{i=1}^m\|W_i\|.
\end{equation}
However, as already mentioned, this constant is usually
quite loose and of limited use to assess the actual stability of 
the network. A novelty of our approach is to take into account the
averagedness properties of the individual activation operators 
to capture more sharply the overall interactions between the 
layers, yielding tighter constants than those provided by 
computing bounds for groups of layers. Our specific 
contributions are the following:
\begin{itemize}
\item 
We show that the most common activation operators used in neural
networks are averaged operators. This not only provides an a 
posteriori justification for Model~\ref{m:1}, but also indicates
that this highly structured framework should be of interest in the 
analysis of other properties of layered networks beyond stability. 
\item 
We derive a general expression for a Lipschitz constant of $T$ in
terms of the averagedness constants of the activation operators 
$(R_i)_{1\leq i\leq m}$ and the norms of certain compositions of 
the linear operators $(W_i)_{1\leq i\leq m}$. 
This Lipschitz constant is shown to lie between the simple upper 
bound \eqref{e:trivubound} and the lower bound 
$\|W_m\circ\cdots\circ W_1\|$ corresponding to a purely
linear network. Our analysis applies to any type of linear 
operator, in particular convolutive ones, and it does not require
any additional assumptions on the activation operator. In
particular, differentiability is not assumed and our results
therefore cover, in particular, networks using the rectified linear 
unit (ReLU) and max-pooling operations.
\item 
In the common situation when the activation operators are
separable, we obtain tighter Lipschitz constants 
for various norms.
\item 
Under some positivity condition, we prove that a 
Lipschitz constant of the network reduces to that of the 
associated purely linear network obtained by removing the nonlinear
operators.
\end{itemize} 
In \cite{2019}, we investigated the special case of
Model~\ref{m:1} in which the activation operators 
$(R_i)_{1\leq i\leq m}$ are proximity operators, hence 
$1/2$-averaged (see Section~\ref{sec:-2}). 
The objective there was to study the asymptotic behavior
of deep network structures rather than their stability.

The remainder of the paper is organized as follows. 
In Section~\ref{sec:gli} we present an illustration of our main
result in a simple special case. 
In Section~\ref{sec:-2} we provide the necessary nonlinear
analysis background. 
In Section~\ref{sec:actav} we show that a wide array of activation
operators used in neural networks are indeed nonexpansive.
In Section~\ref{sec:3} we derive general results
concerning Lipschitz constants for Model~\ref{m:1}. 
Section~\ref{sec:4} refines
this analysis in the case of separable activation operators.

\section{Preview of the main results in a simple scenario} 
\label{sec:gli}
We illustrate on a simple instance the main results of 
the paper. More precisely, we consider a three-layer ($m=3$) 
network where, for every $i\in \{0,1,2,3\}$, $\HH_i$ is
the standard Euclidean space $\RR^{N_i}$. In this case, each
linear operator $W_i$ is identified with a matrix in 
$\RR^{N_i\times N_{i-1}}$. To further simplify our setting, we
assume that the operators $R_1$, $R_2$, and $R_3$
correspond to ReLU layers, that is, for each $i\in\{1,2,3\}$,
\begin{equation}
\big(\forall x=(\xi_{k})_{1\leq k\leq N_i})\in\RR^{N_i}\big)
\quad
R_ix=\big(\rho(\xi_{k})\big)_{1\leq k\leq N_i},
\quad\text{where}\quad
\rho\colon\xi\mapsto\max\{0,\xi\}. 
\end{equation}
In view of \eqref{e:averaged}, $\rho=(1/2)\Id+(1/2)|\cdot|$
is $1/2$-averaged since $|\cdot|$ has Lipschitz constant 1. 
This implies that the operators $R_1$, $R_2$, and $R_3$ 
are also $1/2$-averaged \cite{2019}.
Let us now introduce two parameters which 
will play a central role in our analysis, namely,
\begin{equation}
\label{e:defthetamac} 
\theta_3=\frac{1}{4}\big(\|W_3W_2W_1\|+\|W_3W_2\|\,\|W_1\|
+\|W_3\|\,\|W_2W_1\|+\|W_3\|\,\|W_2\|\,\|W_1\|\big)
\end{equation}
and
\begin{equation}
\vartheta_3=
\sup_{\Lambda_1\in\mathscr{D}_{\{0,1\}}^{(N_1)},
\Lambda_2\in\mathscr{D}_{\{0,1\}}^{(N_2)}}
\|W_3\Lambda_2W_2\Lambda_1W_1\|,
\end{equation}
where $\|\cdot\|$ is the spectral norm and, for each $i\in \{1,2\}$,
$\mathscr{D}_{\{0,1\}}^{(N_i)}$ denotes the set of $N_i\times N_i$
diagonal matrices with entries in $\{0,1\}$. In this context, our
main result states that both $\theta_3$ and $\vartheta_3$ are 
Lipschitz constants of the network, and that
\begin{equation}
\label{e:allsum}
\|W_3 W_{2} W_1\|\leq\vartheta_3\leq\theta_3\leq
\|W_3\|\,\|W_2\|\,\|W_1\|.
\end{equation}
In addition, if the entries of the matrices $(W_i)_{1\leq i\leq 3}$ 
are in $\RP$, then a Lipschitz constant of the network is 
$\|W_3W_2W_1\|$.

\begin{example}
\label{ex:numeric}
\rm
To illustrate the improvement of the proposed bound over the
classical product norm estimate, 
we consider a fully connected network 
with $N_0=8$, $N_1=10$, $N_2=6$, and $N_3=3$.
The entries of the matrices 
$(W_i)_{1\leq i\leq 3}$ are generated randomly and 
independently according
to a normal distribution. We evaluate the Lipschitz constant
estimate $\theta_3$ provided by \eqref{e:defthetamac} and the lower
bound in \eqref{e:allsum}. The average (resp.\ minimal)
value of $\theta_3/(\|W_1\|\,\|W_2\|\,\|W_3\|)$ computed over 1000
realizations is approximately equal to $0.6699$ (resp. $0.5112 $),
while the average (resp.\ minimal) value  of 
$\|W_3W_2W_1\|/(\|W_1\|\,\|W_2\|\,\|W_3\|)$ is approximately equal
to $0.3747$ (resp. $0.1208$). 
In addition, the average
(resp.\ minimal) value of $\vartheta_3/(\|W_1\|\,\|W_2\|\,\|W_3\|)$ 
computed over 1000 realizations is approximately equal
to $0.4528$ (resp.\ $0.2424$). In agreement with 
\eqref{e:allsum}, this estimation of the Lipschitz
constant is better than $\theta_3$ and significantly sharper than
$\|W_1\|\,\|W_2\|\,\|W_3\|$.  
\end{example}

In the remainder of this paper, we show that the above results hold
in a much more general context (for an arbitrary number of 
layers $m$, arbitrary Hilbert spaces, and a
wide class of activation operators), and that some of them can be
extended to non-Euclidean norms. To establish these results, we
need to introduce suitable mathematical tools in the next
section.

\section{Nonexpansive averaged activation operators}
\label{sec:2}

\subsection{Nonlinear analysis tools and notation}
\label{sec:-2}

We review some key facts and definitions which will be used
subsequently; see \cite{Livre1} for further information.
Throughout, $\HH$ is a real Hilbert space with power set $2^\HH$,
scalar product $\scal{\cdot}{\cdot}$, and associated norm 
$\|\cdot\|$.

Let $R\colon\HH\to\HH$ be an operator and let $\alpha\in[0,1]$. 
Then $R$ is nonexpansive if it is $1$-Lipschitzian,
$\alpha$-averaged if there exists a nonexpansive 
operator $Q\colon\HH\to\HH$ such that
$R=(1-\alpha)\Id+\alpha Q$, and firmly nonexpansive if it is
$1/2$-averaged.
Let $A\colon\HH\to 2^{\HH}$ be a set-valued operator. We denote by
$\gra A=\menge{(x,u)\in\HH\times\HH}{u\in Ax}$ the graph of 
$A$ and by $A^{-1}$ the inverse of $A$, i.e., the operator 
with graph $\menge{(u,x)\in\HH\times\HH}{u\in Ax}$.
In addition, $A$ is monotone if
\begin{equation}
(\forall (x,u)\in\gra A)(\forall (y,v)\in\gra A)\quad
\scal{x-y}{u-v}\geq 0,
\end{equation}
and maximally monotone if there exists no monotone operator 
$B\colon\HH\to 2^{\HH}$ such that $\gra A\subset\gra B\neq\gra A$.
If $A$ is maximally monotone, then its resolvent 
$J_A=(\Id+A)^{-1}$ is firmly nonexpansive.
We denote by $\Gamma_0(\HH)$ the class of proper lower 
semicontinuous convex functions from $\HH$ to $\RX$.
Let $f\in\Gamma_0(\HH)$. The conjugate of $f$ is 
\begin{equation}
\Gamma_0(\HH)\ni f^*\colon u\mapsto
\sup_{x\in\HH}(\scal{x}{u}-f(x))
\end{equation}
and the subdifferential of $f$ is
the maximally monotone operator 
\begin{equation}
\partial f\colon\HH\to 2^{\HH}\colon x\mapsto
\menge{u\in\HH}{(\forall y\in\HH)\;\:\scal{y-x}{u}+f(x)\leq f(y)}.
\end{equation}
For every $x\in\HH$, the unique minimizer
of $f+\|x-\cdot\|^2/2$ is denoted by $\prox_f x$.
We have $\prox_f=J_{\partial f}$ and $\prox_f$ is therefore firmly
nonexpansive. 

Let $C$ be a nonempty convex subset of $\HH$. Then
$\iota_C$ is the indicator function of $C$ (it takes values 
$0$ on $C$ and $\pinf$ on its complement) and 
$d_C\colon x\mapsto\min_{y\in C}\|x-y\|$ is its distance function.
If $C$ is closed, its projection operator is 
$\proj_C=\prox_{\iota_C}$.

\subsection{Activators as averaged operators}
\label{sec:actav}
We show via various illustrations that 
the assumption made in Model~\ref{m:1} on the activation
operators covers many existing instances of feed-forward neural
networks. Let us start with some key properties.

\begin{proposition}
\label{p:23}
Let $\HH$ be a real Hilbert space, let $\alpha\in\czeroun$, and 
let $R\colon\HH\to\HH$ be $\alpha$-averaged. Then the following 
hold:
\begin{enumerate}
\itemsep0mm 
\item
\label{p:23i}
There exist a maximally monotone operator $A\colon\HH\to 2^{\HH}$
and a constant $\lambda\in\left[0,2\right]$ such that 
$R=\Id+\lambda(J_A-\Id)$. Furthermore, if
$\lambda\leq 1$, then $R$ is firmly nonexpansive.
\item
\label{p:23ii}
Suppose that $\HH=\RR$. Then there exist a function
$\phi\in\Gamma_0(\RR)$ and a constant $\lambda\in\left[0,2\right]$ 
such that $R=\Id+\lambda(\prox_\phi-\Id)$. Furthermore,
$R$ is increasing if $\lambda\leq 1$ and $R$ is odd if $\phi$ is 
even.
\item
\label{p:23iii}
Suppose that $\HH=\RR$ and that $R$ is increasing. Then there
exists $\phi\in\Gamma_0(\RR)$ such that $R=\prox_\phi$.
\end{enumerate}
\end{proposition}

Next, we illustrate the pervasiveness of nonexpansive averaged 
activation operators in practice, starting with activation
operators on the real line. 

\begin{figure}[!th]
\begin{center}
\vskip -22mm
\scalebox{1.3} 
{
\begin{pspicture}(-3.5,-2.2)(3.5,3.8) 
\def\muu{8.0/(3.0*sqrt(3.0))}
\psplot[plotpoints=400,linewidth=0.025cm,linestyle=solid,%
algebraic,linecolor=blue]{-1.5395}{1.5395}%
{\muu*ATAN(sqrt(abs(x)/(\muu-abs(x))))-%
sqrt(abs(x)*(\muu-abs(x)))-x^2/2}
\psline[linewidth=0.03cm,linecolor=blue,linestyle=solid]{)-}%
(1.5396,2.0)(3.0,2.0)
\psline[linewidth=0.03cm,linecolor=blue,linestyle=solid]{-(}%
(-3.0,2.0)(-1.5396,2.0)
\psline[linewidth=0.02cm,arrowsize=0.05cm 4.0,%
arrowlength=1.4,arrowinset=0.4]{->}(-3,0)(3.0,0)
\psline[linewidth=0.02cm,arrowsize=0.05cm 4.0,%
arrowlength=1.4,arrowinset=0.4]{->}(0,-1)(0,2.0)
\rput(-1.5396,-0.0){\tiny$|$}
\rput(-1.5396,-0.3){\tiny$-\mu$}
\rput(1.5396,-0.0){\tiny$|$}
\rput(1.5396,-0.3){\tiny$\mu$}
\rput(-0.7,1.2332){\tiny$\dfrac{\mu}{2}(\pi-\mu)$}
\rput(-0.0,1.2332){\tiny$-$}
\rput(1.5396,1.2332){\blue\tiny$\bullet$}
\rput(-1.5396,1.2332){\blue\tiny$\bullet$}
\rput(0.0,2.2){\tiny$\phi(x)$}
\rput(3.1,0){\tiny$x$}
\rput(3.3,2){\tiny$\pinf$}
\rput(-3.3,2){\tiny$\pinf$}
\end{pspicture} 
}
\vskip -12mm
\scalebox{0.7} 
{
\begin{pspicture}(-6.7,-3.4)(6.6,4.8) 
\psplot[plotpoints=400,linewidth=0.05cm,linestyle=solid,%
algebraic,linecolor=blue]{-6.0}{6.0}%
{8.0/(3.0*sqrt(3.0))*x*abs(x)/(1.0+x^(2.0))}
\psplot[plotpoints=400,linewidth=0.05cm,linestyle=solid,%
algebraic,linecolor=red]{-6.0}{6.0}%
{(1-1.5)*x+1.5*8.0/(3.0*sqrt(3.0))*x*abs(x)/(1.0+x^(2.0))}
\psplot[plotpoints=400,linewidth=0.05cm,linestyle=solid,%
algebraic,linecolor=dgreen]{-6.0}{6.0}%
{(1-0.5)*x+0.5*8.0/(3.0*sqrt(3.0))*x*abs(x)/(1.0+x^(2.0))}
\psline[linewidth=0.04cm,arrowsize=0.09cm 4.0,%
arrowlength=1.4,arrowinset=0.4]{->}(-6,0)(6.0,0)
\psline[linewidth=0.04cm,arrowsize=0.09cm 4.0,%
arrowlength=1.4,arrowinset=0.4]{->}(0,-4)(0,4.0)
\rput(-2.00,-0.0){$|$}
\rput(-4.00,-0.0){$|$}
\rput(-2.00,-0.4){$-2$}
\rput(-4.00,-0.4){$-4$}
\rput(2.00,-0.0){$|$}
\rput(4.00,-0.0){$|$}
\rput(2.00,-0.4){$2$}
\rput(-0.4,3.00){$3$}
\rput(-0.0,3.00){$-$}
\rput(-0.4,-3.00){$-3$}
\rput(-0.0,-3.00){$-$}
\rput(-0.4,-1.00){$-1$}
\rput(-0.0,-1.00){$-$}
\rput(-0.4,1.00){$1$}
\rput(-0.0,1.00){$-$}
\rput(4.00,-0.4){$4$}
\rput(0.0,4.4){$R(x)$}
\rput(6.3,0){$x$}
\end{pspicture} 
}
\end{center}
\vskip 0mm
\caption{Averaged activation functions: Illustration of 
Example~\ref{ex:23}\ref{ex:23ii}.
Top: The function $\phi$ of \eqref{e:eh8}.
Bottom: In blue, the activation operator $R$ of \eqref{e:14R} 
is the proximity operator of $\phi$, which corresponds to 
$\lambda=1$ in \eqref{e:rprox}. The green curve corresponds to 
the case when $\lambda=0.5$ in \eqref{e:rprox}, and the red 
one to the case when $\lambda=1.5$. As stated in 
Proposition~\ref{p:23}\ref{p:23i}, relaxation parameters 
$\lambda\in[0,1]$ yield increasing activation functions.
Non-monotonic averaged activation functions in \eqref{e:rprox} 
must be generated with relaxation parameters 
$\lambda\in\left]1,2\right]$. 
As seen in Proposition~\ref{p:23}\ref{p:23ii},
since $\phi$ is even, $R$ is odd.}
\label{fig:2}
\end{figure}
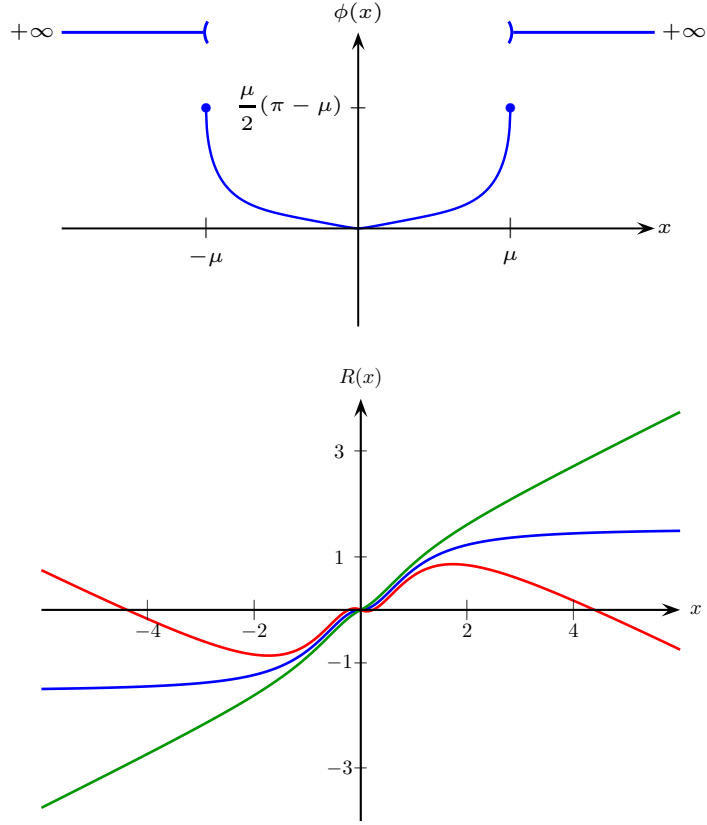

\begin{example}
\label{ex:23}
\rm
Proposition~\ref{p:23}\ref{p:23ii} states that activation functions 
on the real line can be expressed in the generic form
\begin{equation}
\label{e:rprox}
R=\Id+\lambda(\prox_\phi-\Id), 
\quad\text{where}\quad
\phi\in\Gamma_0(\RR) 
\quad\text{and}\quad
\lambda\in[0,2].
\end{equation}
Here are a few explicit instantiations of this proximal
representation.
\begin{enumerate}
\itemsep0mm 
\item 
\label{ex:23i}
If $\lambda=1$, we obtain the class of proximal
activation functions discussed in \cite{2019} and which was seen
there to include standard instances such as 
the unimodal sigmoid activation function \cite[Example~2.13]{2019},
the saturated linear activation function \cite[Example~2.5]{2019},
the ReLU activation function 
\cite[Example~2.6]{2019}, the inverse square root unit activation 
function \cite[Example~2.9]{2019}, the hyperbolic tangent 
activation function \cite[Example~2.12]{2019},
and the Elliot activation function \cite[Example~2.15]{2019}. 
Additional examples in this category are the following. 
Given $\beta\in\RPP$, the capped ReLU activation function 
\cite{Kriz10} is 
\begin{equation}
(\forall x\in\RR)\quad R(x)=\prox_{\iota_{[0,\beta]}}(x)
=\min\{\max\{x,0\},\beta\}, 
\end{equation}
and, for $\beta\leq 1$, the exponential linear unit (ELU) function 
\cite{Clev15} is 
\begin{equation}
\label{e:elu}
(\forall x\in\RR)\quad R(x)=
\begin{cases}
x,&\text{if}\;\;x\geq 0;\\
\beta\big(\exp(x)-1\big), &\text{if}\;\;x<0.
\end{cases}
\end{equation}
It follows from \cite[Cor.~24.5, Prop.~24.32, and
Exa.~13.2(v)]{Livre1} that $R=\prox_{\phi}$, where
\begin{equation}
\label{e:elu2}
(\forall x\in\RR)\quad 
\phi(x)=
\begin{cases}
0&\text{if}\;\;x\geq 0;\\
(x+\beta)\ln\bigg(\dfrac{x+\beta}{\beta}\bigg)
-x-\dfrac{x^2}{2},
&\text{if}\;\;-\beta<x<0;\\
\beta-\dfrac{\beta^2}{2},&\text{if}\;\;x=-\beta;\\
\pinf,&\text{if}\;\;x<-\beta.
\end{cases}
\end{equation}
The softplus activation function \cite{Glor11} 
$R\colon x\mapsto\ln((1+e^x)/2)$ is also a proximity operator since
it is nonexpansive and increasing (see 
Proposition~\ref{p:23}\ref{p:23iii}).
\item
\label{ex:23ii}
The Geman--McClure function
\cite{Gema85} 
\begin{equation}
\label{e:14R}
(\forall x\in\RR)\quad 
R(x)=\frac{\mu\operatorname{sign}(x)x^2}{1+x^2},
\quad\text{where}\quad\mu=\dfrac{8}{3\sqrt{3}},
\end{equation}
will be employed in Example~\ref{ex:404usa}.
Set $\psi=|\cdot|-\arctan|\cdot|\in\Gamma_0(\RR)$.
Then $R$ is nonexpansive and $R=\mu\psi'$.
The conjugate of $\mu\psi$ is 1-strongly convex and given by
$\mu\psi^*(\cdot/\mu)$, where 
\begin{equation}
\label{e:psibarcaps}
(\forall x\in\RR)\quad\psi^*(x)=   
\begin{cases}
\arctan\sqrt{\dfrac{|x|}{1-|x|}}-\sqrt{|x|(1-|x|)},
&\text{if}\;\;|x|<1;\\
\dfrac{\pi}{2},&\text{if}\;\;|x|=1;\\
\pinf,&\text{otherwise.}
\end{cases}
\end{equation}
It follows from \cite[Cor.~24.5]{Livre1} that $R=\prox_{\phi}$
with (see Fig.~\ref{fig:2})
\begin{equation}
\label{e:eh8}
\phi=\mu\psi^*\bigg(\dfrac{\cdot}{\mu}\bigg)-
\dfrac{|\cdot|^2}{2}\colon x\mapsto 
\begin{cases}
\mu\arctan\sqrt{\dfrac{|x|}{\mu-|x|}}-\sqrt{|x|(\mu-|x|)}
-\dfrac{x^2}{2},&\text{if}\;\;|x|<\mu;\\
\dfrac{\mu(\pi-\mu)}{2},&\text{if}\;\;|x|=\mu;\\
\pinf,&\text{otherwise.}
\end{cases}
\end{equation}
\item
Take $\phi=\iota_{\RP}$. Then we obtain 
the leaky ReLU activation function \cite{Maas13} for $0<\lambda<1$,
the ReLU activation function for $\lambda=1$, and the absolute 
value activation function \cite{Brun13} for $\lambda=2$.
\item
\label{ex:23iv}
The use of nonmonotonic activation functions has been advocated 
in various papers. They turn out to be $\alpha$-averaged
(alternatively, in view of Proposition~\ref{p:23}\ref{p:23ii}, 
they are of the form \eqref{e:rprox} with 
$\lambda\in\left]1,2\right]$). 
To compute the averagedness constant of a nonexpansive operator
$R\colon\RR\to\RR$, one can proceed as follows.
According to \eqref{e:averaged}, we must find the smallest 
$\alpha\in\rzeroun$ such that $Q=\Id+\alpha^{-1}(R-\Id)$ remains
nonexpansive. This means that 
the supremum of the modulus of the one-sided derivatives 
(the derivatives if they exist) over $\RR$ should be one.
Thus, we obtain $\alpha=1$ for the sine activation function 
$R=\sin$ \cite{Naga95}, as well as for
the absolute value function $R=|\cdot|$
\cite{Brun13} and the mirrored ReLU activation function
\cite{Zhao16}
\begin{equation}
(\forall x\in\RR)\quad R(x)=\proj_{[0,1]}|x|=
\begin{cases}
|x|,&\text{if}\;\;|x|<1;\\
1,& \text{otherwise,}
\end{cases}
\end{equation}
$\alpha\approx 0.546$ for the swish activation 
function \cite{Rama17}
\begin{equation}
(\forall x\in\RR)\quad R(x)=\dfrac{10x}{11(1+\exp(-x))},
\end{equation}
$\alpha\approx 0.536$ for the exponential linear squashing (ELiSH) 
function \cite{Basi18}
\begin{equation}
(\forall x\in\RR)\quad R(x)=\frac{10}{11}\times 
\begin{cases}
\dfrac{x}{1+\exp(-x)},&\text{if}\;\;x\geq 0;\\[4mm]
\dfrac{\exp(x)-1}{1+\exp(-x)},&\text{if}\;\;x<0,
\end{cases}
\end{equation}
and $\alpha=(1+\sqrt{2/e})/2$ for the Gaussian activation function 
$R\colon x\mapsto\exp(-x^2)$ \cite{Mhas94}.
\end{enumerate}
\end{example}

Next, is a technique for lifting a proximal activation 
operator from $\RR$ to a Hilbert space $\HH$.

\begin{example}
\label{ex:404usa}
\rm
Let $\HH$ be a real Hilbert space, let $\lambda\in[0,2]$, let 
$C$ be a nonempty closed convex subset of $\HH$, let 
$\phi\in\Gamma_0(\RR)$ be an even function such that 
$\phi^*$ is differentiable on $\RR\smallsetminus\{0\}$ with $0$ as
its unique minimizer. Set
\begin{equation}
\label{e:304usa}
(\forall x\in\HH)\quad
Rx=
\begin{cases}
(1-\lambda)x+
\Frac{\lambda\prox_{\phi}d_C(x)}{d_C(x)}
(x-\proj_Cx),&\text{if}\;\;x\notin C;\\
(1-\lambda)x,&\text{if}\;\;x\in C.
\end{cases}
\end{equation}
Then $R$ is $\lambda/2$-averaged.
In particular, set
$\lambda=1$, $C=\{0\}$, $\mu=8/(3\sqrt{3})$ and 
define $\phi$ as in \eqref{e:eh8}. Then 
we infer that the squashing function 
\begin{equation}
\label{e:404usa}
R\colon x\mapsto\frac{\mu\|x\|}{1+\|x\|^2}x
\end{equation}
used in capsule networks \cite{Sabo17}
is a proximal activation operator.
\end{example}

Another construction that builds on activation functions on the 
real line is the following, which is reminiscent of the original
multilayer perceptrons \cite{Rose58}.

\begin{example}
\label{ex:93}
\rm
Let $\HH$ be a separable real Hilbert space, let
$\emp\neq\bK\subset\NN$, let $(e_k)_{k\in\bK}$ be an
orthonormal basis of $\HH$, and let $\alpha\in [0,1]$.
For every $k\in\bK$, let $\varrho_k\colon\RR\to\RR$
be $\alpha$-averaged and such that $\varrho_k (0)=0$. Define
$R\colon\HH\to\HH\colon
x\mapsto\sum_{k\in\bK}\varrho_k(\scal{x}{e_k})e_k$.
Then $R$ is $\alpha$-averaged. 
\end{example}

\begin{example}
\label{ex:rt5}
\rm
Let $N$ be a strictly positive integer, let $\omega\in\czeroun$,
and let $C$ be a nonempty closed convex subset of $\RR^N$. Set
\begin{equation}
R\colon\RR^N\to\RR^N\colon 
(\xi_{k})_{1\leq k\leq N}\mapsto
\omega\big(\xi_{k}^\uparrow\big)_{1\leq k\leq N}+(1-\omega) 
\proj_C(\xi_{k})_{1\leq k\leq N},
\end{equation}
where $(\xi_{k}^\uparrow)_{1\leq k\leq N}$ denotes the vector 
obtained by sorting the components of $(\xi_{k})_{1\leq k\leq N}$
in ascending order. Then $R$ is $(1+\omega)/2$-averaged.
\end{example}

\begin{remark}
\rm
\label{r:l5}
Set $C=\menge{(\xi_k)_{1\leq k\leq N}\in\RR^N}
{\xi_1=\cdots=\xi_N}$ in Example~\ref{ex:rt5}.
Then
\begin{equation}
R\colon\RR^N\to \RR^N \colon 
(\xi_{k})_{1\leq k\leq N}\mapsto
\Bigg(\omega\xi_{k}^\uparrow+\frac{1-\omega}{N}
\sum_{j=1}^N\xi_{j}\Bigg)_{1\leq k\leq N}.
\end{equation}
Now set $W\colon\RR^N\to\RR\colon(\xi_{k})_{1\leq k\leq N}
\mapsto\xi_N$.
Then $W\circ R$ corresponds to the max-average pooling
performed on a block of size $N$ \cite{Leec16}. 
When $\omega=0$, the standard average-pooling operation is obtained,
which is associated with the activation operator $\proj_C$.
When $\omega=1$, we recover the standard max-pooling operation
\cite{Bour10}, which is the main building block of maxout layers
\cite{Good13}. The max-pooling operator is nonexpansive.
\end{remark}

\begin{example}
\label{ex:rt6}
\rm
Let $2\leq N\in\NN$, let 
$\{\tau_k\}_{1\leq k\leq N-1}\subset\left]-1,1\right[$, 
and let $\theta\in\RR$. Set
\begin{equation}
R\colon\RR^{N-1}\to\RR^{N-1}\colon 
(\xi_{k})_{1\leq k\leq N-1}\mapsto US\Big(
[\tau_1\xi_1,\ldots,\tau_{N-1}\xi_{N-1},\theta]^\top\Big),
\end{equation}
where $U\in\RR^{(N-1)\times N}$ is the matrix obtained 
by retaining the first $(N-1)$
rows of the identity matrix of size $N\times N$, and 
$S\colon\RR^N\to\RR^N\colon(\xi_{k})_{1\leq k\leq N}
\mapsto (\xi_{k}^\uparrow)_{1\leq k\leq N}$. Then $R$ is 
$(1+\max\{|\tau_1|,\ldots,|\tau_{N-1}|\})/2$-averaged.
\end{example}

\begin{remark} 
\rm
Let $N\geq 3$ be an odd integer, 
let $(\tau_k)_{1\leq k\leq N-1}\in\left]-1,1\right[^{N-1}$,  
let $\theta\in\RR$, let $R$ be the activation operator defined in
Example~\ref{ex:rt6}, and set
$W\colon \RR^{N-1}\to\RR\colon (\xi_{k})_{1\leq k\leq N-1}\mapsto
\xi_{\frac{N+1}{2}}$.
Then, for every $x=(\xi_{k})_{1\leq k\leq N-1}\in \RR^{N-1}$, 
$(W\circ R)x=\text{\rm median}\{\tau_1\xi_1,\ldots,\tau_{N-1}
\xi_{N-1},\theta\}$.
This corresponds to the median neuron model introduced in
\cite{Alad14}.
\end{remark}

\begin{remark}
\rm
Multi-component averaged activation operators can be
derived from the\linebreak above examples. Indeed, let
$(\HH_i)_{1\leq i\leq M}$ be real Hilbert spaces and let
$\HH=\HH_1\oplus\cdots\oplus\HH_M$ be their Hilbert direct sum. 
For every $i\in\{1,\ldots,M\}$, let $\alpha_i\in[0,1]$ and let
$R_i\colon\HH_i\to\HH_i$ be $\alpha_i$-averaged. Then
$R\colon\HH\to\HH\colon (x_i)_{1\leq i\leq M}\mapsto
(R_ix_i)_{1\leq i\leq M}$ is $\alpha$-averaged with
$\alpha=\max_{1\leq i\leq M}\alpha_i$.
\end{remark}

\section{Lipschitz constants for layered networks}
\label{sec:3}

The objective of this section is to derive Lipschitz constants 
for networks conforming to Model~\ref{m:1}.
Note that, if $m=1$, a Lipschitz constant of $T$ is
clearly $\theta_1=\|W_1\|$ since $R_1$ is nonexpansive. 
We shall therefore focus henceforth on the case $m\geq 2$. 
Throughout, the following notation is employed.

\begin{notation}
Let $2\leq m\in\NN$ and $k\in\{1,\ldots,m-1\}$. Then
\begin{equation}
\JJ_{m,k}=
\begin{cases}
\menge{(j_1,\ldots,j_k)\in\NN^k}{1\leq j_1<\cdots<j_k\leq m-1},
&\text{if}\;\;k>1;\\
\{1,\ldots,m-1\}, &\text{if}\;\; k=1
\end{cases}
\end{equation}
and, for every $(j_1,\ldots,j_k)\in\JJ_{m,k}$,
\begin{equation}
\label{e:wcal}
\sigma_{m;\{j_1,\ldots,j_k\}}=\|W_m\circ\cdots\circ W_{j_k+1}\|\,
\|W_{j_k}\circ\cdots\circ W_{j_{k-1}+1}\|\cdots 
\|W_{j_1}\circ\cdots\circ W_1\|.
\end{equation}
\end{notation}

\begin{theorem}
\label{t:propLipgen}
Consider the setting of Model~\ref{m:1} with $m\geq 2$. Set 
\begin{equation}
\label{e:beta}
(\forall\,\JJ\subset \{1,\ldots,m-1\})
\quad\beta_{m;\JJ} =\Bigg(\prod_{j\in\JJ}\alpha_j\Bigg)
\prod_{j\in\{1,\ldots,m-1\}\smallsetminus\JJ}(1-\alpha_j)
\end{equation}
and 
\begin{equation}
\label{e:defthetam}
\theta_m=\beta_{m;\emp}\|W_m\circ\cdots\circ W_1\|
+\sum_{k=1}^{m-1}
\sum_{(j_1,\ldots,j_k)\in\JJ_{m,k}} \beta_{m;\{j_1,\ldots,j_k\}}
\sigma_{m;\{j_1,\ldots,j_k\}}.
\end{equation}
Then $\theta_m$ is a Lipschitz constant of $T$.
\end{theorem}

The following proposition features some important special cases.
\begin{proposition}
\label{p:1}
Consider the setting of Model~\ref{m:1} with $m\geq 2$, and let
$\theta_m$ be defined as in \eqref{e:defthetam}. Then the following
hold:
\begin{enumerate}
\itemsep0mm 
\item 
\label{p:1ii} 
$\|W_m\circ\cdots\circ W_1\|\leq
\theta_m\leq\prod_{i=1}^m\|W_i\|$.
\item 
\label{p:1iii} 
Suppose that, for every $i\in\{1,\ldots,m-1\}$, $R_i=\Id$. Then 
$\theta_m=\|W_m\circ\cdots\circ W_1\|$.
\item 
\label{p:1iv} 
Suppose that, for every $i\in\{1,\ldots,m-1\}$, $R_i$ is purely
nonexpansive in the sense that $\alpha_i=1$ is its smallest 
averaging constant. Then $\theta_m=\prod_{i=1}^m\|W_i\|$.
\item 
\label{p:1v} 
Suppose that, for every $i\in\{1,\ldots,m-1\}$, $R_i$ is firmly
nonexpansive. Then
\begin{equation}
\label{e:defthetama}
\theta_m=\frac1{2^{m-1}} 
\Bigg(\|W_m\circ\cdots\circ W_1\|
+\sum_{k=1}^{m-1}\sum_{(j_1,\ldots,j_k)\in\JJ_{m,k}} 
\sigma_{m;\{j_1,\ldots,j_k\}}\Bigg).
\end{equation}
\item 
\label{p:1vi} 
Set $\alpha_0=\theta_0=1$. Then 
\begin{equation}
\label{e:thetamrec}
\theta_m=\sum_{i=0}^{m-1}\alpha_i\theta_i 
\Bigg(\prod_{q=i+1}^{m-1}(1-\alpha_{q})\Bigg)\|W_m\circ\cdots\circ
W_{i+1}\|. 
\end{equation}
\end{enumerate}
\end{proposition}

\begin{remark}\
\rm
Proposition~\ref{p:1}\ref{p:1ii}--\ref{p:1}\ref{p:1iv} show that
the tightest bound in terms of stability corresponds to a
linear network, while the loosest corresponds to a network with
nonlinearities having no stronger property than nonexpansiveness.
\end{remark}

We close this section by observing that the Lipschitz constant
exhibited in Theorem~\ref{t:propLipgen} is a componentwise
increasing function of the averagedness constants of 
the activation operators.

\begin{proposition}
\label{p:alphagencond}
Consider the setting of Model~\ref{m:1} with $m\geq 2$. Make the 
Lipschitz constant $\theta_m$ in Theorem~\ref{t:propLipgen} a 
function of $(\alpha_1,\ldots,\alpha_{m-1})\in [0,1]^{m-1}$.
Let $(\alpha_i)_{1\leq i\leq m-1}\in [0,1]^{m-1}$
and
$(\alpha'_i)_{1\leq i\leq m-1}\in [0,1]^{m-1}$
be such that
$(\forall i\in\{1,\ldots,m-1\})$ $\alpha_i \leq\alpha'_i$.
Then $\theta_m(\alpha_1,\ldots,\alpha_{m-1})\leq
\theta_m(\alpha'_1,\ldots,\alpha'_{m-1})$.
\end{proposition}

\begin{remark}
\rm
Proposition~\ref{p:alphagencond} suggests that, in terms of 
stability, it is better
to use proximal activation operators, such as those listed in
Example \ref{ex:23}\ref{ex:23i}--\ref{ex:23ii}, than
$\alpha$-averaged activation operators for which $\alpha>1/2$, 
such as those mentioned in Example~\ref{ex:23}\ref{ex:23iv}.
\end{remark}

\section{Networks using separable activation operators}
\label{sec:4}

We show that sharper Lipschitz constants can be
derived in the case of networks featuring the type of separable
structure described in Example~\ref{ex:93}. 
Note that this class of networks is the most commonly used, 
standard convnets being special cases.
The following notation will be used.

\begin{notation}
Let $\HH$ be a separable real Hilbert space, let
$\emp\neq\bK\subset\NN$, let $E=(e_k)_{k\in\bK}$ be an
orthonormal basis of $\HH$, and let $I$ be a nonempty
bounded subset of $\RR$. Then
\begin{equation}
\label{e:defLambda}
\mathscr{D}_I(E)=\Menge{\Lambda\colon\HH\to\HH\colon
x\mapsto\sum_{k\in\bK}\lambda_k\scal{x}{e_k}e_k}{
\{\lambda_k\}_{k\in\bK}\subset I}.
\end{equation}
\end{notation}

\subsection{General results}

\begin{theorem}
\label{t:sep}
Consider the setting of Model~\ref{m:1} with $m\geq 2$. For 
every $i\in\{1,\ldots,m-1\}$, suppose that $\HH_i$ is separable,
let $\emp\neq\bK_i\subset\NN$, let ${E_i}=(e_{i,k})_{k\in\bK_i}$ 
be an orthonormal basis of $\HH_i$, and,
for every $k\in\bK_i$, let $\varrho_{i,k}\colon\RR\to\RR$ be 
$\alpha_i$-averaged and such that $\varrho_{i,k}(0)=0$.
Assume that
\begin{equation}
\label{e:stgermain1}
(\forall i\in\{1,\ldots,m-1\})\quad
R_i\colon\HH_i\to\HH_i\colon x\mapsto\sum_{k\in\bK_i}
\varrho_{i,k}(\scal{x}{e_{i,k}})e_{i,k}
\end{equation}
and define
\begin{equation}
\label{e:thetambin}
\vartheta_m= 
\sup_{\substack{\Lambda_1\in
\mathscr{D}_{\{1-2\alpha_1,1\}}(E_1)\\\vdots\\\Lambda_{m-1}\in
\mathscr{D}_{\{1-2\alpha_{m-1},1\}}(E_{m-1})}} \|W_m\circ
\Lambda_{m-1}\circ\cdots\circ \Lambda_1\circ W_1\|.
\end{equation}
Then the following hold:
\begin{enumerate}
\itemsep0mm 
\item
\label{t:sepi}
$\vartheta_m$ is a Lipschitz constant of the operator $T$ of
\eqref{e:defTi}.
\item
\label{t:sepii}
Define $\theta_m$ as in \eqref{e:defthetam}. Then
$\|W_m\circ\cdots\circ W_1\|\leq\vartheta_m\leq \theta_m$.
\end{enumerate}
\end{theorem}

\begin{remark}
\rm
An expression similar to \eqref{e:thetambin} was proposed 
empirically in \cite{Scam18} for a multilayer perceptron 
operating on finite-dimensional spaces
under the additional assumption that the activation 
operators are continuously differentiable 
and firmly nonexpansive.
\end{remark}

\begin{remark}
\rm
\label{r:stgermain1}
In Theorem~\ref{t:sep}, make the additional assumption that, 
for some\linebreak $i\in\{1,\ldots,m-1\}$, the functions
$(\varrho_{i,k})_{k\in\bK_i}$ are increasing. Then it follows from 
Proposition~\ref{p:23}\ref{p:23iii} that there exist functions 
$(\phi_{i,k})_{k\in\bK_i}$ in $\Gamma_0(\RR)$ 
such that $(\forall k\in\bK_i)$ $\varrho_{i,k}=\prox_{\phi_{i,k}}$. 
In addition, for every $k\in\bK_i$, since $\varrho_{i,k}(0)=0$
and since the set of minimizers of $\phi_{i,k}$ coincides with the
set of fixed points of $\prox_{\phi_{i,k}}$
\cite[Proposition~12.29]{Livre1}, we deduce that 
$\phi_{i,k}$ is minimized at $0$.
Furthermore, 
$\alpha_i=1/2$ and $R_i=\prox_{\varphi_i}$, where 
$(\forall x\in\HH_i)$
$\varphi_i(x)=\sum_{k\in\bK_i}\phi_{i,k}(\scal{x}{e_{i,k}})$.
Such a construction is used in \cite{Siop07,Smms05}.
\end{remark}

As in Proposition~\ref{p:alphagencond}, the Lipschitz constant 
exhibited in Theorem~\ref{t:sep} turns out to be a componentwise
increasing function of the averagedness constants of the activation
operators.

\begin{proposition}
\label{p:35}
Consider the setting of Model~\ref{m:1} with $m\geq 2$. For every 
$i\in\{1,\ldots,m-1\}$, suppose that $\HH_i$ is separable, let 
$\emp\neq\bK_i\subset\NN$, and let ${E_i}=(e_{i,k})_{k\in\bK_i}$ 
be an orthonormal basis of $\HH_i$. Define 
$\vartheta_m\colon [0,1]^{m-1}\to\RP$ by
\begin{equation}
\label{e:thetambinvar}
\vartheta_m\colon 
(\alpha_1,\ldots,\alpha_{m-1})\mapsto
\sup_{\substack{\Lambda_1\in
\mathscr{D}_{\{1-2\alpha_1,1\}}(E_1)\\\vdots\\\Lambda_{m-1}\in
\mathscr{D}_{\{1-2\alpha_{m-1},1\}}(E_{m-1})}} \|W_m\circ
\Lambda_{m-1}\circ\cdots\circ \Lambda_1\circ W_1\|.
\end{equation}
Let $(\alpha_i)_{1\leq i\leq m-1}\in [0,1]^{m-1}$ and
$(\alpha'_i)_{1\leq i\leq m-1}\in [0,1]^{m-1}$
be such that
$(\forall i\in\{1,\ldots,m-1\})$ $\alpha_i \leq\alpha'_i$.
Then $\vartheta_m(\alpha_1,\ldots,\alpha_{m-1})\leq
\vartheta_m(\alpha'_1,\ldots,\alpha'_{m-1})$.
\end{proposition}

\subsection{Extension to non-Hilbertian norms}

In certain applications, Hilbertian norms may not be the most 
relevant measures to quantify errors. We now state a variant of 
Theorem~\ref{t:sep} which holds for alternative norms.
It involves embeddings of Hilbert spaces;
standard examples can be found in \cite{Zeid19}. Let us also point
out that these embedding conditions are automatically satisfied if
the spaces are finite-dimensional.

\begin{proposition}
\label{p:sepbis}
Consider the setting of Model~\ref{m:1} with $m\geq 2$.
For every $i\in\{1,\ldots,m\}$, suppose that $\HH_i$ is separable,
let $\emp\neq\bK_i\subset\NN$, let ${E_i}=(e_{i,k})_{k\in\bK_i}$ 
be an orthonormal basis of $\HH_i$, and,
for every $k\in\bK_i$, let $\varrho_{i,k}\colon\RR\to\RR$ be 
$\alpha_i$-averaged and such that $\varrho_{i,k}(0)=0$.
Let $\GG_0$ be the normed space obtained by equipping 
the vector space underlying $\HH_0$ with a norm for which 
$\GG_0$ is continuously embedded in $\HH_0$, and let $\GG_m$ be
the normed space obtained by equipping the vector space 
underlying $\HH_m$ with a norm
for which $\HH_m$ is continuously embedded in $\GG_m$.
Assume that 
\begin{equation}
(\forall i\in\{1,\ldots,m\})\quad
R_i\colon\HH_i\to\HH_i\colon
x\mapsto\sum_{k\in\bK_i}
\varrho_{i,k}\big(\scal{x}{e_{i,k}}\big)e_{i,k}.
\end{equation}
Then 
\begin{equation}
\label{e:thetambinbis}
\vartheta_m= 
\sup_{\substack{\Lambda_1\in\mathscr{D}_{\{1-2\alpha_1,1\}}
(E_1)\\\vdots\\
\Lambda_m\in\mathscr{D}_{\{1-2\alpha_m,1\}}(E_m)}} 
\|\Lambda_m\circ W_m
\circ\cdots\circ \Lambda_1\circ W_1\|_{\GG_0,\GG_m}
\end{equation} 
is a Lipschitz constant of $T\colon\GG_0\to\GG_m$.
\end{proposition}

\begin{corollary}
\label{c:sepbis}
Consider the setting of Model~\ref{m:1} with $m\geq 2$. Define
$\GG_0$ and $(R_i)_{1\leq i\leq m}$ as in 
Proposition~\ref{p:sepbis}, let $p\in\left[1,\pinf\right]$, and 
let $\{\omega_k\}_{k\in\bK_m}\subset\RPP$ be 
such that one of the following holds:
\begin{enumerate}
\itemsep0mm 
\item
$p\in\left[1,2\right[$ and 
$\sum_{k\in\bK_m}\omega_k^{2/(2-p)}<\pinf$.
\item
$p\in\left[2,\pinf\right]$ and $\sup_{k\in\bK_m}\omega_k<\pinf$.
\end{enumerate}
Let $\GG_m$ be the normed space obtained by equipping the 
vector space underlying $\HH_m$ with the norm
\begin{equation}
\label{e:normlp}
\big(\forall x\in\HH_m\big)\quad
\|x\|_{\GG_m}=
\begin{cases}
\displaystyle\Bigg|\sum_{k\in\bK_m}
\omega_k|\scal{x}{e_{m,k}}|^{p}\Bigg|^{1/p},
&\text{if}\;\;p<\pinf;\\
{\displaystyle\sup_{k\in\bK_m}}\omega_k|\scal{x}{e_{m,k}}|,
&\text{if}\;\;p=\pinf.
\end{cases}
\end{equation}
Then a Lipschitz constant of $T\colon\GG_0\to\GG_m$ is
\begin{equation}
\label{e:dl231}
\vartheta_m=\sup_{\substack{\Lambda_1\in 
\mathscr{D}_{\{1-2\alpha_1,1\}}(E_1)\\
\vdots\\
\Lambda_{m-1}\in\mathscr{D}_{\{1-2\alpha_{m-1},1\}}(E_{m-1})}} 
\|W_m\circ \Lambda_{m-1}\circ\cdots\circ\Lambda_1
\circ W_1\|_{\GG_0,\GG_m}.
\end{equation}
\end{corollary}

\subsection{Networks with positive weights}

Under certain positivity assumptions, the constant 
$\vartheta_m$ of \eqref{e:thetambin} and 
\eqref{e:dl231} can be simplified.

\begin{assumption}
\label{as:netpos}
Consider the setting of Model~\ref{m:1} with $m\geq 2$. For every 
$i\in\{0,\ldots,m\}$, suppose that $\HH_i$ is separable, let 
$\emp\neq\bK_i\subset\NN$, and let ${E_i}=(e_{i,k})_{k\in\bK_i}$ 
be an orthonormal basis of $\HH_i$. 
For every $(k_0,\ldots,k_m)\in
\bK_0\times\cdots\times\bK_m$, set
\begin{equation}
\label{e:elt}
\mu_{k_0,\ldots,k_m}=\scal{W_1e_{0,k_0}}{e_{1,k_1}}
\cdots\scal{W_me_{m-1,k_{m-1}}}{e_{m,k_m}}.
\end{equation}
We suppose that
\begin{multline} 
\label{e:condpos}
(\forall (k_0,\ldots,k_m)\in\bK_0\times\cdots\times\bK_m)
(\forall (l_0,\ldots,l_{m-1})\in
\bK_0\times\cdots\times\bK_{m-1})\\
\mu_{k_0,\ldots,k_{m-1},k_m} \,
\mu_{l_0,\ldots,l_{m-1},k_m} \geq 0.
\end{multline}
\end{assumption}

\begin{example}
\label{ex:can}
\rm
Consider the particular case of Model~\ref{m:1} in which, 
for every $i\in\{0,\ldots,m\}$, $N_i\in\NN\smallsetminus\{0\}$,
$\HH_i=\RR^{N_i}$, $E_i$ is the canonical basis of $\RR^{N_i}$
and, for every $k\in\{1,\ldots,N_i\}$,
$\chi_{i,k}\in\{-1,1\}$ with the additional condition that,
for every $l\in\{1,\ldots,N_0\}$, $\chi_{0,k}=\chi_{0,l}$.
Further, for every 
$i\in\{1,\ldots,m\}$, the matrix 
$W_i=[w_{i,k,l}]_{1\leq k \leq N_i,1\leq l \leq N_{i-1}}
\in\RR^{N_i\times N_{i-1}}$ satisfies
\begin{equation}
(\forall k\in\{1,\ldots,N_i\})
(\forall l\in\{1,\ldots,N_{i-1}\})\quad
w_{i,k,l}=\chi_{i,k}\chi_{i-1,l} |w_{i,k,l}|.
\end{equation}
Then Assumption~\ref{as:netpos} holds. This is true in particular
if, for every $i\in\{1,\ldots,m\}$,\linebreak
$\{w_{i,k,l}\}_{1\leq k\leq N_i,1\leq l\leq N_{i-1}}\subset\RP$,
which corresponds to positively weighted networks. 
See \cite{Chor15} for the design of such networks.
\end{example}

In the following result, 
a Lipschitz constant of the network \eqref{e:defTi} coincides
with that of the linear network $W_m\circ\cdots\circ W_1$ for
standard choices of norms.

\begin{proposition}
\label{p:sepbispos}
Suppose that the assumptions of 
Corollary~\ref{c:sepbis} are satisfied, that 
\begin{equation}
\label{e:symnorm0}
\big(\forall(\xi_{k})_{k\in \bK_0}\in\ell^2(\bK_0)\big)
\quad\Bigg\|\sum_{k\in \bK_0} \xi_{k}e_{0,k}\Bigg\|_{\GG_0} =
\Bigg\|\sum_{k\in \bK_0} |\xi_{k}| e_{0,k}\Bigg\|_{\GG_0},
\end{equation}
and that Assumption~\ref{as:netpos} holds. Then the Lipschitz
constant $\vartheta_m$ of $T\colon\GG_0\to\GG_m$ in \eqref{e:dl231}
reduces to 
$\vartheta_m=\|W_m\circ\cdots\circ W_1\|_{\GG_0,\GG_m}$.
\end{proposition}

We show below that the Lipschitz constant of a positively weighted
network associated with weight operators $(W_i)_{1\leq i\leq m}$ 
and nonseparable activation operators is not necessarily
$\|W_m\circ\cdots\circ W_1\|$.

\begin{example}
\label{ex:tanh}
\rm
Consider the toy version of Model~\ref{m:1} in which $m=2$, 
$\HH_0=\HH_1=\HH_2=\RR^2$. Set
$\varphi\colon x=(\xi_1,\xi_2)\mapsto\phi(\xi_1)+ \phi(\xi_2)$,
where 
\begin{align}
\phi\colon\RR\to\RX\colon
\xi\mapsto
\begin{cases}
\dfrac{(1+\xi)\ln(1+\xi)+(1-\xi)\ln(1-\xi)-\xi^2}{2}
&\text{if}\;\;|\xi|<1;\\
\ln(2)-1/2&\text{if}\;\;|\xi|=1;\\
\pinf,&\text{if}\;\;|\xi|>1.
\end{cases}
\end{align}
Let $\xi\in\left]-1,1\right[=\dom\phi'=\dom(\Id+\phi')
=\ran\prox_\phi$. Then $\xi+\phi'(\xi)=\text{arctanh}(\xi)$ and 
therefore $\varrho=(\Id+\phi')^{-1}=\text{tanh}$. Consequently,
we derive from \cite[Example~2.13]{2019} that 
$(\forall x=(\xi_1,\xi_2)\in\RR^2)$
$\prox_{\varphi}x=\big(\text{tanh}(\xi_1),\text{tanh}(\xi_2)\big)$.
Now set
\begin{equation}
b_1=b_2=0,\quad
U=\frac12 
\begin{bmatrix}
\sqrt{3} & 1\\
1 & -\sqrt{3}
\end{bmatrix},\quad
W_1=
\begin{bmatrix}
1 & 3\\
3 & 3
\end{bmatrix},
\quad
W_2=\begin{bmatrix}
10 & 2\\
7 & 4
\end{bmatrix},
\end{equation}
$R_1=\prox_{\varphi\circ U}=U\circ\prox_{\varphi}\circ U$
\cite[Lemma~2.8]{Smms05}, and
$R_2=\Id$. Then $\|W_2W_1\|\approx 54.72$.
If the input $x= (-3.4,2)$ is perturbed by 
$z=10^{-4}\times(1,\sqrt{3})$, we get
${\|T(x+z)-Tx\|}/{\|z\|}\approx 58.18$,
which shows that, although $W_1$ and $W_2$ have strictly positive
entries, the Lipschitz constant is larger than $\|W_2W_1\|$.
Note that, in this scenario, the constant of \eqref{e:defthetam} is 
\begin{equation}
\theta_2=(\|W_2W_1\|+\|W_2\|\|W_1\|)/2\approx 60.50. 
\end{equation}
A sharper
Lipschitz constant can be obtained by noticing that this
network is equivalent to a network
in which $W_1$, $W_2$, and $R_1$ are replaced
by $W_1'=U W_1$, $W_2'= W_2U$, and
$R_1'=\prox_{\varphi}$. Since $R_1'$ is separable, the constant
of \eqref{e:thetambinvar} is $\vartheta_2\approx 59.54$. 
In contrast, the 
naive bound of \eqref{e:trivubound} 
is about $66.29$.
\end{example}

For separable activators in finite-dimensional spaces, we have the 
following result, which does not require Assumption~\ref{as:netpos}.

\begin{proposition}
\label{p:sepbisposb}
Consider the setting of Model~\ref{m:1} with $m\geq 2$.
Suppose that the assumptions of Corollary~\ref{c:sepbis}  hold
and that $\|\cdot\|_{\GG_0}$ satisfies \eqref{e:symnorm0}.
In addition, assume that, for every $i\in\{0,\ldots,m\}$, 
$\HH_i=\RR^{N_i}$ and $E_i$ is the canonical basis of $\RR^{N_i}$.
For every $i\in\{1,\ldots,m\}$, let $A_i$ denote the
matrix obtained by taking the absolute values of the
entries of the matrix $W_i$. Then the Lipschitz constant 
$\vartheta_m$ of $T\colon\GG_0\to\GG_m$ in \eqref{e:dl231} satisfies 
$\vartheta_m\leq\|A_m\cdots A_1\|_{\GG_0,\GG_m}$.
\end{proposition}

\section{Conclusion}
Using advanced tools from nonlinear analysis, we have derived sharp
Lipschitz constants for layered network structures involving
compositions of nonexpansive averaged operators and affine
operators. This framework has been shown to model feed-forward
neural networks having a chain graph structure. Extending these
results to networks having a more general dyadic acyclic graph
(DAG) structure would be of interest. Among the many avenues of
future research that this work suggests, it would be interesting to
exploit it to devise training strategies to achieve better
robustness. The proposed nonexpansive operator machinery could also
be used to design network architectures with smaller Lipschitz
constants. Finally, computing local Lipschitz constants could be of
interest in practice and constitutes an important topic
of future research.
\appendix

\section{Technical lemmas}

\begin{lemma}{\rm\cite[Proposition~2.4]{Siop07}}
\label{l:prox24}
Let $R$ be a function defined from $\RR$ to $\RR$. Then $R$ is 
the proximity operator of a function in $\Gamma_0(\RR)$ if 
and only if it is nonexpansive and increasing. 
\end{lemma}

\begin{lemma}
\label{l:sup}
Let $q\in\NN\smallsetminus\{0\}$ and, for every 
$i\in\{1,\ldots,q\}$, 
let $S_i$ be a nonempty subset of a real vector space $\XX_i$. 
Let $\psi\colon\XX_1\times\cdots\times\XX_q\to\RR$ be a
function which is convex with respect to each of its $q$
coordinates. Set $\boldsymbol{S}=S_1\times\cdots\times S_q$ and let
$\conv\boldsymbol{S}$ be its convex envelope. Then
$\sup\psi(\boldsymbol{S})=\sup\psi(\conv\boldsymbol{S})$.  
\end{lemma}
\begin{proof}
Set $\mu=\sup\psi(\boldsymbol{S})$. Clearly,
$\mu\leq\sup\psi(\conv\boldsymbol{S})$. 
Now take $\boldsymbol{x}\in\conv\boldsymbol{S}$. Then
$\boldsymbol{x}=\sum_{j\in I}\alpha_j\boldsymbol{x}_j$, where
$(\alpha_j)_{j\in I}$ is a finite family in $\rzeroun$ such that
$\sum_{j\in I}\alpha_j=1$ and, for every $j\in I$, 
$\boldsymbol{x}_j=(x_{j,i})_{1\leq i\leq q}$, with
$(\forall i\in\{1,\ldots,q\})$ $x_{j,i}\in S_i$. Note that
$(\forall(j_1,\ldots,j_q)\in I^q)$ 
$(x_{j_1,1},\ldots,x_{j_q,q})\in\boldsymbol{S}$. Therefore, 
\begin{align}
\psi(\boldsymbol{x})
&=\psi\bigg(\sum_{j_1\in I}\alpha_{j_1}x_{j_1,1},\ldots,
\sum_{j_q\in I}\alpha_{j_q}x_{j_q,q}\bigg)\nonumber\\
&\leq\sum_{j_1\in I}\alpha_{j_1}\psi\bigg(x_{j_1,1},
\sum_{j_2\in I}\alpha_{j_2}x_{j_2,2},\ldots,
\sum_{j_q\in I}\alpha_{j_q}x_{j_q,q}\bigg)\nonumber\\
&\;\;\vdots\nonumber\\[-2mm]
&\leq\sum_{j_1\in I,\ldots,j_q\in I}\bigg(\prod_{i=1}^q
\alpha_{j_i}\bigg)\psi\big(x_{j_1,1},
\ldots,x_{j_q,q}\big)
\leq\mu.
\end{align}
Hence, $\sup\psi(\conv\boldsymbol{S})=
\sup_{\boldsymbol{x}\in\conv\boldsymbol{S}}\psi(\boldsymbol{x})
\leq\mu$.
\end{proof}

\begin{lemma}
\label{l:alphascal}
Let $\HH$ be a separable real Hilbert space, let
$\emp\neq\bK\subset\NN$, let $E=(e_k)_{k\in\bK}$ be an
orthonormal basis of $\HH$, and let $\alpha\in [0,1]$.
For every $k\in\bK$, let $\varrho_k\colon\RR\to\RR$
be $\alpha$-averaged and such that $\varrho_k (0)=0$. Define
$R\colon\HH\to\HH\colon
x\mapsto\sum_{k\in\bK}\varrho_k(\scal{x}{e_k})e_k$, and fix
$x$ and $y$ in $\HH$. Then there exists 
$\Lambda\in\mathscr{D}_{[1-2\alpha,1]}({E})$ such that 
$Rx-Ry=\Lambda (x-y)$.
\end{lemma}
\begin{proof}
We saw in Example~\ref{ex:93} that $R$ is well defined.
We have 
\begin{equation}
\label{e:RxRy}
Rx-Ry=\sum_{k\in\bK}\big(\varrho_k(\scal{x}{e_k})-
\varrho_k(\scal{y}{e_k})\big)e_k.
\end{equation}
For every $k\in\bK$, there exists a 
nonexpansive $\theta_k\colon\RR\to\RR$ such that
$\varrho_k=(1-\alpha)\Id+\alpha\theta_k$ and, therefore,
\begin{equation}
\varrho_k(\scal{x}{e_k})-\varrho_k(\scal{y}{e_k})
=(1-\alpha)(\scal{x}{e_k}-\scal{y}{e_k})
+\alpha\big(\theta_k(\scal{x}{e_k})-\theta_k(\scal{y}{e_k})\big).
\end{equation}
Consequently, for every $k\in\bK$, 
there exists $\lambda_k\in [1-2\alpha,1]$ such that
\begin{equation}
(1-\alpha)\big(\scal{x}{e_k}-\scal{y}{e_k}\big)
+\alpha\big(\theta_k(\scal{x}{e_k})-\theta_k(\scal{y}{e_k})\big)
=\lambda_k(\scal{x}{e_k}-\scal{y}{e_k}).
\end{equation}
We deduce from \eqref{e:RxRy} that
$Rx-Ry=\sum_{k\in\bK}\lambda_k (\scal{x}{e_k}-\scal{y}{e_k})e_k$,
as claimed.
\end{proof}

\section{Proofs of main results}
\subsection{Proof of Proposition~\ref{p:23}}
\ref{p:23i}:
As seen in \eqref{e:averaged}, there exists a nonexpansive operator
$Q\colon\HH\to\HH$ such that $R=(1-\alpha)\Id+\alpha Q$. However, 
by \cite[Prop.~4.4 and Cor.~23.9]{Livre1}, there 
exists a maximally monotone operator
$A\colon\HH\to 2^{\HH}$ such that $Q=2J_A-\Id$.
Hence, $R=\Id+\lambda(J_A-\Id)$, where
$\lambda=2\alpha\in\left[0,2\right]$. For the last claim, 
notice that, since $J_A$ is firmly nonexpansive
\cite[Cor.~23.9]{Livre1}, so is 
$R=(1-\lambda)\Id+\lambda J_A$ as a convex combination of two
firmly nonexpansive operators \cite[Exa.~4.7]{Livre1}. 
\ref{p:23ii}$\Rightarrow$\ref{p:23i}: It follows from 
\cite[Cor.~22.23]{Livre1} that there exists 
$\phi\in\Gamma_0(\RR)$ such that
$A=\partial\phi$, which provides the expression for $R$. 
The increasingness claim follows from Lemma~\ref{l:prox24}.
Finally, if $\phi$ is even, then $\prox_\phi$ is odd
\cite[Prop.~24.10]{Livre1} and so is $R$. 
\ref{p:23iii}: This follows from Lemma~\ref{l:prox24}.

\subsection{Proof of Example~\ref{ex:404usa}}
Let $\sigma_C$ be the support function of $C$ and set 
$f=\sigma_C+\phi\circ\|\cdot\|\in\Gamma_0(\HH)$. Then it 
follows from \cite[Prop.~24.30]{Livre1} and \eqref{e:304usa} 
that $R=\Id+\lambda(\prox_f-\Id)$, However, since $\prox_f$ is 
firmly nonexpansive, it is $1/2$-averaged, which makes $R$
a $\lambda/2$-averaged operator. Now consider the function $\phi$ of
\eqref{e:eh8}. Then it is an even function in $\Gamma_0(\RR)$
with $0$ as its unique minimizer. Next, set 
$\psi=|\cdot|-\arctan|\cdot|$. As seen in 
Example~\ref{ex:23}\ref{ex:23ii}, 
$\phi=\mu\psi^*(\cdot/\mu)-|\cdot|^2/2$ and 
$\dom\psi^*$ is bounded. Therefore $\dom\phi=\mu\dom\psi^*$ 
is bounded. In turn, $\phi$ is supercoercive and we derive from 
\cite[Prop.~14.15]{Livre1} that $\dom\phi^*=\RR$.
Hence, since $\phi=\phi^{**}$ is strictly convex, it follows
derive from \cite[Prop.~18.9]{Livre1} that $\phi^*$ is 
differentiable on $\RR$. In addition, $d_C=\|\cdot\|$. Altogether, 
\eqref{e:304usa} reduces to 
\begin{equation}
(\forall x\in\HH)\quad
Rx=
\begin{cases}
\Frac{\prox_{\phi}\|x\|}{\|x\|}
x,&\text{if}\;\;x\neq 0;\\
0,&\text{if}\;\;x=0
\end{cases}
\end{equation}
and hence, in view of Example~\ref{ex:23}\ref{ex:23ii}, to 
\eqref{e:404usa}.

\subsection{Proof of Example~\ref{ex:93}}
Let $x\in\HH$ and $y\in\HH$.
It follows from the nonexpansiveness of the functions 
$(\varrho_k)_{k\in\bK}$ that
\begin{equation}
\sum_{k\in\bK}\big|\varrho_k(\scal{x}{e_k})\big|^2
=\sum_{k\in\bK}\big|\varrho_k(\scal{x}{e_k})-\varrho_k(0)\big|^2
\leq\sum_{k\in\bK}\big|\scal{x}{e_k}-0\big|^2=\|x\|^2.
\end{equation}
Hence, $R$ is well defined. For every $k\in\bK$, 
by \eqref{e:averaged} there exists a 
nonexpansive function $\theta_k\colon\RR\to\RR$ such that
$\varrho_k=(1-\alpha)\Id+\alpha\theta_k$. Hence, 
$Rx=(1-\alpha) x+\alpha Qx$, where
$Qx=\sum_{k\in\bK}\theta_k(\scal{x}{e_k})e_k$.
Therefore,
\begin{equation}
\|Qx-Qy\|^2=\sum_{k\in\bK}
\big|\theta_k(\scal{x}{e_k})-\theta_k(\scal{y}{e_k})\big|^2
\leq\sum_{k\in\bK}\big|\scal{x}{e_k}-\scal{y}{e_k}\big|^2
=\|x-y\|^2.
\end{equation}
This shows that $Q$ is nonexpansive and hence that
$R$ is $\alpha$-averaged.

\subsection{Proof of Example~\ref{ex:rt5}}
Let $S$ be the sorting operator of Example~\ref{ex:rt6}. Then
\begin{align}
(\forall x\in\RR^N)(\forall y\in\RR^N)\;\;\|Sx-Sy\|^2 
&=\|S x\|^2-2\scal{Sx}{Sy}+\|S y\|^2\nonumber\\
&=\|x\|^2-2\scal{Sx}{Sy}+\|y\|^2\nonumber\\
&\leq\|x\|^2-2\scal{x}{y}+\|y\|^2\label{e:ineqrea}\\
&=\|x-y\|^2,
\label{e:nU7}
\end{align}
where \eqref{e:ineqrea} follows from \cite[Thm.~368]{Hard52}.
This shows that $S$ is nonexpansive. Furthermore, 
$Q=2\proj_C-\Id$ is nonexpansive \cite[Cor.~4.18]{Livre1}. 
Note that
\begin{equation}
(1-\omega)\proj_C+\omega S 
=\bigg(1-\frac{1+\omega}{2}\bigg)\Id+\frac{1+\omega}{2}\left(
\frac{1-\omega}{1+\omega}Q+\frac{2\omega}{1+\omega}S\right).
\end{equation}
Since $((1-\omega)Q+2\omega S)/(1+\omega)$ is
nonexpansive as a convex combination of nonexpansive operators, the
operator $(1-\omega)\proj_C+\omega S$ is $(1+\omega)/2$-averaged. 

\subsection{Proof of Example~\ref{ex:rt6}}
Set $A=\operatorname{Diag}(\tau_1,\ldots,\tau_{N-1})$. 
Let $x$ and $y$ be in $\RR^{N-1}$, and define
$\widetilde{x}=[(Ax)^\top,\theta]^\top$ and
$\widetilde{y}=[(Ay)^\top,\theta]^\top$. As seen in 
\eqref{e:nU7}, $S$ is nonexpansive. Consequently, 
\begin{align}
\|Rx-Ry\| 
&=\|US \widetilde{x}-US \widetilde{y}\|
\leq \|U\|\,\|S \widetilde{x}-S\widetilde{y}\|
=\|S\widetilde{x}-S\widetilde{y}\|\nonumber\\
&\leq\|\widetilde{x}-\widetilde{y}\|
=\|Ax-Ay\|
\leq\max\{|\tau_1|,\ldots,|\tau_{N-1}|\} \|x-y\|.
\end{align}
This shows that $R$ is Lipschitzian with constant
$\max\{|\tau_1|,\ldots,|\tau_{N-1}|\}<1$. It is thus 
$\alpha$-averaged with 
$\alpha=(1+\max\{|\tau_1|,\ldots,|\tau_{N-1}|\})/2$ 
\cite[Prop.~4.38]{Livre1}.

\subsection{Proof of Theorem~\ref{t:propLipgen}}
For every $i\in\{1,\ldots,m\}$,
$P_i=R_i(\cdot+b_i)$ is $\alpha_i$-averaged
and, therefore, there exists 
a nonexpansive operator $Q_i\colon\HH_i\to\HH_i$ such that
$P_i=(1-\alpha_i)\Id+\alpha_iQ_i$.
Since $T=P_m\circ W_m\circ\cdots\circ P_1\circ W_1$ and $P_m$ is
nonexpansive, it suffices to show that 
\begin{equation}
\label{e:recurrence}
\theta_m\;\text{is a Lipschitz constant of}\;
W_m\circ\cdots\circ P_1\circ W_1. 
\end{equation}
Let us prove this result by induction. Let $x\in\HH_0$ and
$y\in\HH_0$. If $m=2$, we derive from the
nonexpansiveness of $Q_1$ that
\begin{align}
\label{e:m=2}
&\hskip -8mm
\|(W_2\circ P_1\circ W_1)x-(W_2\circ P_1\circ W_1)y\|\nonumber\\
&=\|(W_2\circ((1-\alpha_1)\Id+\alpha_1 Q_1)\circ W_1)x-
(W_2\circ ((1-\alpha_1)\Id+\alpha_1 Q_1)\circ W_1)y\|\nonumber\\
&\leq(1-\alpha_1)\|(W_2\circ W_1)(x-y)\|+
\alpha_1\|(W_2\circ Q_1\circ W_1)x
-(W_2\circ Q_1\circ W_1)y\|\nonumber\\
&\leq(1-\alpha_1)\|W_2\circ W_1\|\,\|x-y\|
+\alpha_1\|W_2\|\,\|Q_1(W_1x)-Q_1(W_1y)\|\nonumber\\
&\leq(1-\alpha_1)\|W_2\circ W_1\|\,\|x-y\|
+\alpha_1\|W_2\|\,\|W_1(x-y)\|\nonumber\\
&\leq\big((1-\alpha_1)\|W_2\circ W_1\|+
\alpha_1\|W_2\|\,\|W_1\|\big)\|x-y\|.
\end{align}
Hence, $T$ is Lipschitzian with constant 
\begin{equation}
(1-\alpha_1)\|W_2\circ W_1\|+\alpha_1\|W_2\|\,\|W_1\|
=\beta_{2;\emp}\|W_2\circ W_1\|+ \beta_{2;\{1\}} \|W_2\|\,\|W_1\| =
\theta_2.
\end{equation}
Now assume that $m>2$ and that \eqref{e:recurrence} holds at order
$m-1$. Then
\begin{align*}
&\|(W_m\circ P_{m-1}\circ\cdots\circ P_1\circ W_1)x
-(W_m\circ P_{m-1}\circ\cdots\circ P_1\circ W_1)y\|\nonumber\\
&=\|(W_m\circ ((1-\alpha_{m-1})\Id+\alpha_{m-1} Q_{m-1})\circ
\cdots\circ P_1\circ W_1)x\nonumber\\
&\quad\;-(W_m\circ ((1-\alpha_{m-1})\Id+\alpha_{m-1} Q_{m-1})\circ
\cdots\circ P_1\circ W_1)y\|\nonumber\\
&\leq(1-\alpha_{m-1})\|(W_m\circ W_{m-1}\circ\cdots\circ
P_1\circ W_1)x-(W_m\circ W_{m-1}\circ\cdots\circ
P_1\circ W_1)y\|\nonumber\\
&\quad\;+\alpha_{m-1}\|(W_m\circ Q_{m-1}\circ W_{m-1}\circ
\cdots\circ P_1\circ W_1)x-(W_m\circ Q_{m-1}\circ W_{m-1}\circ
\cdots\circ P_1\circ W_1)y\|\nonumber\\
&\leq(1-\alpha_{m-1})\|(W_m\circ W_{m-1}\circ\cdots\circ
P_1\circ W_1)x-(W_m\circ W_{m-1}\circ\cdots\circ
P_1\circ W_1)y\|\nonumber\\
&\quad+\alpha_{m-1}\|W_m\|\,\|(Q_{m-1}\circ W_{m-1}\circ \cdots
\circ P_1\circ W_1)x-(Q_{m-1}\circ W_{m-1}\circ\cdots\circ
P_1\circ W_1)y\|.
\end{align*}
Hence, the nonexpansiveness of $Q_{m-1}$ yields
\begin{align}
&\|(W_m\circ P_{m-1}\circ\cdots\circ P_1\circ W_1)x
-(W_m\circ P_{m-1}\circ\cdots\circ P_1\circ W_1)y\|\nonumber\\
&\leq (1-\alpha_{m-1})\|(W_m\circ W_{m-1}\circ P_{m-2}\circ
\cdots\circ  W_1)x-(W_m\circ W_{m-1}\circ
P_{m-2}\circ\cdots\circ W_1)y\|\nonumber\\
&\quad\;+\alpha_{m-1}\|W_m\|\,\|(W_{m-1}\circ P_{m-2}\circ \cdots
\circ P_1\circ W_1)x-(W_{m-1}\circ P_{m-2}\circ\cdots\circ
P_1\circ W_1)y\|.\nonumber\\
\label{e:21}
\end{align}
On the other hand, the induction hypothesis yields
\begin{align}
&
\|(W_{m-1}\circ P_{m-2}\circ\cdots\circ P_1\circ W_1)x-
(W_{m-1}\circ P_{m-2}\circ\cdots\circ P_1\circ W_1)y\|\nonumber\\
&\!\!\!\leq\theta_{m-1}\|x-y\|\nonumber\\
&\!\!\!=\bigg(\beta_{m-1;\emp}\|W_{m-1}\circ\cdots\circ W_1\|
+\sum_{k=1}^{m-2}\sum_{(j_1,\ldots,j_k)\in\JJ_{m-1,k}}
\beta_{m-1;\{j_1,\ldots,j_k\}}\sigma_{m-1;\{j_1,\ldots,j_k\}}
\bigg)\|x-y\|.\nonumber\\
\label{e:22}
\end{align}
Similarly, replacing $W_{m-1}$ by $W_m\circ W_{m-1}$ above, we get
\begin{align}
&
\|((W_m\circ W_{m-1})\circ P_{m-2}\circ\cdots\circ P_1\circ W_1)x-
((W_m\circ W_{m-1})\circ P_{m-2}\circ\cdots\circ P_1\circ W_1)y\|
\nonumber\\
&\leq\bigg(\beta_{m-1;\emp}\|W_m\circ W_{m-1}\circ\cdots\circ
W_1\|+\sum_{k=1}^{m-2}
\sum_{(j_1,\ldots,j_k)\in \JJ_{m-1,k}} 
\beta_{m-1;\{j_1,\ldots,j_k\}}\sigma_{m;\{j_1,\ldots,j_k\}}
\bigg)\|x-y\|.\nonumber\\
\label{e:23}
\end{align}
\vskip -4mm
\noindent
Using \eqref{e:21}, and then inserting \eqref{e:23} and
\eqref{e:22}, we obtain
\begin{align}
&\|(W_m\circ P_{m-1}\circ\cdots\circ P_1\circ W_1)x
-(W_m\circ P_{m-1}\circ\cdots\circ P_1\circ W_1)y\|\nonumber\\
&\leq (1-\alpha_{m-1})\|(W_m\circ W_{m-1}\circ P_{m-2}\circ
\cdots\circ  W_1)x-(W_m\circ W_{m-1}\circ
P_{m-2}\circ\cdots\circ W_1)y\|\nonumber\\
&\quad\;+\alpha_{m-1}\|W_m\|\,\|(W_{m-1}\circ P_{m-2}\circ \cdots
\circ P_1\circ W_1)x-(W_{m-1}\circ P_{m-2}\circ\cdots\circ
P_1\circ W_1)y\|\nonumber\\
&\leq (1-\alpha_{m-1})\times\nonumber\\
&\quad\;
\bigg(\beta_{m-1;\emp}\|W_m\circ W_{m-1}\circ\cdots\circ
W_1\|+\sum_{k=1}^{m-2}
\sum_{(j_1,\ldots,j_k)\in \JJ_{m-1,k}} 
\beta_{m-1;\{j_1,\ldots,j_k\}}\sigma_{m;\{j_1,\ldots,j_k\}}
\bigg)\|x-y\|
\nonumber\\
&\quad\;+\alpha_{m-1}\|W_m\|\times\nonumber\\
&\quad\;
\bigg(\beta_{m-1;\emp}\|W_{m-1}\circ\cdots\circ W_1\|
+\sum_{k=1}^{m-2}\sum_{(j_1,\ldots,j_k)\in\JJ_{m-1,k}}
\beta_{m-1;\{j_1,\ldots,j_k\}}\sigma_{m-1;\{j_1,\ldots,j_k\}}
\bigg)\|x-y\|.
\label{e:24}
\end{align}
Furthermore, we deduce from \eqref{e:beta} that 
\begin{equation}
(\forall\,\JJ\subset\{1,\ldots,m-1\})\quad\beta_{m;\JJ}=
\begin{cases}
(1-\alpha_{m-1})\beta_{m-1;\JJ},
&\text{if}\;\;m-1\not\in \JJ;\\
\alpha_{m-1}\beta_{m-1;\JJ\smallsetminus\{m-1\}},
&\text{if}\;\;m-1\in\JJ.
\end{cases}
\label{e:25}
\end{equation}
Therefore 
\begin{equation}
\label{e:sj88}
\begin{cases}
\beta_{m;\emp}=(1-\alpha_{m-1})\beta_{m-1;\emp}\\
\beta_{m;\{j_1,\ldots,j_k\}}=
(1-\alpha_{m-1})\beta_{m-1;\{j_1,\ldots,j_k\}}
&\text{if}\;\;m-1\notin\{j_1,\ldots,j_k\}\\
\beta_{m;\{j_1,\ldots,j_k\}}=
\alpha_{m-1}\beta_{m-1;\{j_1,\ldots,j_k\}\smallsetminus\{m-1\}}
&\text{if}\;\;m-1\in\{j_1,\ldots,j_k\},
\end{cases}
\end{equation}
which implies that, if $m-1\not\in\{j_1,\ldots,j_k\}$, then
$\beta_{m;\{j_1,\ldots,j_{k},m-1\}}
=\alpha_{m-1}\beta_{m-1;\{j_1,\ldots,j_{k}\}}$.
Hence, \eqref{e:24} yields
\begin{align}
&\hskip -14mm \|(W_m\circ P_{m-1}\circ\cdots\circ P_1\circ W_1)x
-(W_m\circ P_{m-1}\circ\cdots\circ P_1\circ W_1)y\|/\|x-y\|
\nonumber\\
&\leq\beta_{m;\emp}\|W_m\circ W_{m-1}\circ\cdots\circ W_1\|
\nonumber\\
&\quad\;+\sum_{k=1}^{m-2}
\sum_{(j_1,\ldots,j_k)\in \JJ_{m-1,k}} (1-\alpha_{m-1})
\beta_{m-1;\{j_1,\ldots,j_k\}}\sigma_{m;\{j_1,\ldots,j_k\}}
\nonumber\\
&\quad\;+\alpha_{m-1}\beta_{m-1;\emp}
\|W_m\|\,\|W_{m-1}\circ\cdots\circ W_1\|
\nonumber\\
&\quad\;+\sum_{k=1}^{m-2}\sum_{(j_1,\ldots,j_k)\in\JJ_{m-1,k}}
\alpha_{m-1}
\beta_{m-1;\{j_1,\ldots,j_k\}}\|W_m\|\sigma_{m-1;\{j_1,\ldots,j_k\}}
\nonumber\\
&=\beta_{m;\emp}\|W_m\circ W_{m-1}\circ\cdots\circ W_1\|
+\beta_{m;m-1}\sigma_{m;\{m-1\}}
\nonumber\\
&\quad\;+\sum_{k=1}^{m-2}\sum_{(j_1,\ldots,j_k)\in\JJ_{m,k}
\setminus\{m-1\}}
\beta_{m;\{j_1,\ldots,j_k\}}\sigma_{m;\{j_1,\ldots,j_k\}}
\nonumber\\
&\quad\;+\sum_{k=1}^{m-2}\sum_{(j_1,\ldots,j_k)\in\JJ_{m-1,k}}
\beta_{m;\{j_1,\ldots,j_k,m-1\}}\sigma_{m;\{j_1,\ldots,j_k,m-1\}}
\nonumber\\
&=\beta_{m;\emp}\|W_m\circ \cdots\circ W_1\|
+\sum_{j=1}^{m} \beta_{m;\{j\}} \sigma_{m;\{j\}}
\nonumber\\
&\quad\;+
\sum_{k=2}^{m-2}\sum_{(j_1,\ldots,j_k)\in\JJ_{m,k}\setminus\{m-1\}}
\beta_{m;\{j_1,\ldots,j_k\}}\sigma_{m;\{j_1,\ldots,j_k\}}
\nonumber\\
&\quad\;+\sum_{k=2}^{m-1}
\sum_{\substack{(j_1,\ldots,j_k)\in\JJ_{m,k}\\j_{k}=m-1}}
\beta_{m;\{j_1,\ldots,j_k\}}\sigma_{m;\{j_1,\ldots,j_k\}}
\nonumber\\
&=\beta_{m;\emp}\|W_m\circ\cdots\circ W_1\|
+\sum_{k=1}^{m-1}
\sum_{(j_1,\ldots,j_k)\in\JJ_{m,k}} \beta_{m;\{j_1,\ldots,j_k\}}
\sigma_{m;\{j_1,\ldots,j_k\}}
\nonumber\\
&=\theta_m.
\label{e:26}
\end{align}
Thus, we obtain
\begin{equation}
\label{e:B18}
\|(W_m\circ P_{m-1}\circ\cdots\circ P_1\circ W_1)x-(W_m\circ
P_{m-1}\circ\cdots\circ P_1\circ W_1)y\|\leq\theta_m\|x-y\|,
\end{equation}
which establishes \eqref{e:recurrence}.

\subsection{Proof of Proposition~\ref{p:1}}
Define $(\beta_{m;\JJ})_{\JJ\subset\{1,\ldots,m-1\}}$ as in 
\eqref{e:beta}.
\ref{p:1ii}:
For every $k\in\{1,\ldots,m-1\}$ and every
$(j_1,\ldots,j_k)\in\JJ_{m,k}$, \eqref{e:wcal} yields
\begin{equation}
\|W_m\circ\cdots\circ W_1\| \leq
\sigma_{m;\{j_1,\ldots,j_k\}} \leq\prod_{i=1}^m\|W_i\|.
\end{equation}
Consequently, it follows from \eqref{e:defthetam} that
\begin{equation}
\label{e:bst}
\|W_m\circ\cdots\circ W_1\|
\sum_{\JJ\subset\{1,\ldots,m-1\}}\beta_{m;\JJ}
\leq\theta_m\leq\Bigg(\prod_{i=1}^m\|W_i\|\Bigg)
\sum_{\JJ\subset\{1,\ldots,m-1\}} \beta_{m;\JJ}.
\end{equation}
In view of \eqref{e:beta}, 
$(\beta_{m;\JJ})_{\JJ\subset\{1,\ldots,m-1\}}$ is 
the discrete probability distribution of a vector of $m-1$
independent Bernoulli random variables. Hence,
$\sum_{\JJ\subset\{1,\ldots,m-1\}}\beta_{m;\JJ}=1$
in \eqref{e:bst}. 
\ref{p:1iii}:
For every $i\in\{1,\ldots,m-1\}$, $\alpha_i=0$. Therefore, in 
view of \eqref{e:beta},
\begin{equation}
(\forall\JJ\subset\{1,\ldots,m-1\})
\quad\beta_{m;\JJ}=
\begin{cases}
1,&\text{if}\;\;\JJ=\emp;\\
0,&\text{if}\;\;\JJ\neq\emp.
\end{cases}
\end{equation}
Hence, the result follows from \eqref{e:defthetam}.
\ref{p:1iv}:
For every $i\in\{1,\ldots,m-1\}$, $\alpha_i=1$. Therefore, in 
view of \eqref{e:beta},
\begin{equation}
(\forall\JJ\subset\{1,\ldots,m-1\})\quad\beta_{m;\JJ}=
\begin{cases}
1,&\text{if}\;\;\JJ=\{1,\ldots,m-1\};\\
0,&\text{if}\;\;\JJ\neq\{1,\ldots,m-1\}.
\end{cases}
\end{equation}
Invoking \eqref{e:defthetam} allows us to conclude.
\ref{p:1v}:
For every $i\in\{1,\ldots,m-1\}$ $\alpha_i=1/2$. 
Hence, \eqref{e:beta} yields
$(\forall\JJ\subset\{1,\ldots,m-1\})$
$\beta_{m;\JJ}=2^{1-m}$.
Invoking once again \eqref{e:defthetam} yields the result.
\ref{p:1vi}:
It follows from \eqref{e:wcal} that
\begin{align}
\label{e:m1}
&\hskip -4mm
\sum_{k=1}^{m-1}\sum_{1\leq j_1<\ldots<j_k\leq m-1}
\beta_{m;\{j_1,\ldots,j_k\}}
\sigma_{m;\{j_1,\ldots,j_k\}}\nonumber\\
&=\;\sum_{k=1}^{m-1}\sum_{1\leq j_1<\ldots<j_k\leq m-1}
\beta_{m;\{j_1,\ldots,j_k\}}
\|W_m\circ\cdots\circ W_{j_k+1}\|\,
\|W_{j_k}\circ\cdots\circ W_{j_{k-1}+1}\|\nonumber\\
&\hskip 102mm\cdots \|W_{j_1}\circ\cdots\circ W_1\|.
\end{align}
We decompose this expression in a sum of terms depending on
the value $i$ taken by $j_{k}$, namely,
\begin{align}
&\hskip -4mm
\sum_{k=1}^{m-1}\sum_{1\leq j_1<\ldots<j_k\leq m-1}
\beta_{m;\{j_1,\ldots,j_k\}}
\sigma_{m;\{j_1,\ldots,j_k\}}\nonumber\\
&=\;\sum_{i=1}^{m-1} \beta_{m;\{i\}} 
\|W_m\circ\cdots\circ W_{i+1}\|\,\|W_i\circ\cdots W_1\|\nonumber\\
&\quad+\sum_{k=2}^{i-1}\sum_{1\leq j_1<\ldots<j_{k-1}\leq i-1}
\beta_{m;\{j_1,\ldots,j_{k-1},i\}}
\|W_m\circ\cdots\circ W_{i+1}\|\,
\|W_i\circ\cdots\circ W_{j_{k-1}+1}\|\nonumber\\
&\hskip 102mm\cdots \|W_{j_1}\circ\cdots\circ W_1\|.
\label{e:jaienviederoupiller}
\end{align}
In addition, for every $(j_1,\ldots,j_{k-1})\in\JJ_{i,k-1}$, we
derive from \eqref{e:beta} that
\begin{align}
\beta_{m;\{j_1,\ldots,j_{k-1},i\}}
&=\Bigg(\prod_{j\in\{j_1,\ldots,j_{k-1},i\}}\alpha_j\Bigg)
\prod_{j\in\{1,\ldots,m-1\}\smallsetminus
\{j_1,\ldots,j_{k-1},i\}}(1-\alpha_j)\nonumber\\
&=\alpha_i\Bigg(\prod_{j\in\{j_1,\ldots,j_{k-1}\}}
\alpha_j\Bigg)
\Bigg(\prod_{q=i+1}^{m-1}(1-\alpha_{q})\Bigg) 
\prod_{j\in\{1,\ldots,i-1\}\smallsetminus
\{j_1,\ldots,j_{k-1}\}}(1-\alpha_j)\nonumber\\
&=\alpha_i\Bigg(\prod_{q=i+1}^{m-1}(1-\alpha_{q})\Bigg) 
\beta_{i;\{j_1,\ldots,j_{k-1}\}}.
\end{align}
Using the above equality in \eqref{e:jaienviederoupiller},
factorizing common factors, and invoking \eqref{e:defthetam} yields 
\begin{align}
\label{e:m2}
&\hskip -9mm
\sum_{k=1}^{m-1}\sum_{1\leq j_1<\ldots<j_k\leq m-1}
\beta_{m;\{j_1,\ldots,j_k\}}
\sigma_{m;\{j_1,\ldots,j_k\}}\nonumber\\
&=\; \sum_{i=1}^{m-1} \alpha_i
\Bigg(\prod_{q=i+1}^{m-1}(1-\alpha_{q})\Bigg) 
\|W_m\circ\cdots\circ W_{i+1}\| \Big( \beta_{i;\emp}\|W_i\circ
\cdots W_1\|\nonumber\\
&\quad+ \sum_{k=2}^i\sum_{1\leq j_1<\ldots<j_{k-1}\leq i-1}
\beta_{i;\{j_1,\ldots,j_{k-1}\}}
\|W_i\circ\cdots\circ W_{j_{k-1}+1}\|\cdots 
\|W_{j_1}\circ\cdots\circ W_1\|\Big)\nonumber\\\
&=\; \sum_{i=1}^{m-1}\alpha_i\theta_i
\Bigg(\prod_{q=i+1}^{m-1}(1-\alpha_{q})\Bigg) 
\|W_m\circ\cdots\circ W_{i+1}\|,
\end{align}
and we obtain \eqref{e:thetamrec}.

\subsection{Proof of Proposition~\ref{p:alphagencond}}
Let $l\in\{1,\ldots,m-1\}$ and set
\begin{equation}
\label{e:betamod}
(\forall\,\JJ\subset \{1,\ldots,m-1\}\smallsetminus\{l\})
\quad\beta_{m,l;\JJ}=\Bigg(\prod_{j\in\JJ}\alpha_j\Bigg)
\prod_{j\in\{1,\ldots,m-1\}
\smallsetminus(\JJ\cup\{l\})}(1-\alpha_j).
\end{equation}
For every $k\in \{1,\ldots,m-1\}$ and every
$(j_1,\ldots,j_k)\in\JJ_{m,k}$, \eqref{e:wcal} yields
\begin{equation}
\label{e:ineqcalW}
\sigma_{m;\{j_1,\ldots,j_k\}} \leq 
\sigma_{m;\{j_1,\ldots,j_k\}\cup\{l\}}.
\end{equation}
We infer from \eqref{e:defthetam} that
\begin{align}
\label{e:defthetamm}
&\hskip -3mm\theta_m(\alpha_1,\ldots,\alpha_{m-1})\nonumber\\
&=(1-\alpha_{l})\beta_{m,l;\emp} 
\|W_m\circ\cdots\circ W_1\|+\sum_{k=1}^{m-1}
\sum_{\substack{(j_1,\ldots,j_k)\in\JJ_{m,k}\\ 
l\in\{j_1,\ldots,j_k\}}} 
\alpha_{l}\beta_{m,l;\{j_1,\ldots,j_k\}\smallsetminus\{l\}}
\sigma_{m;\{j_1,\ldots,j_k\}}\nonumber\\
&\quad\;+\sum_{k=1}^{m-2}
\sum_{\substack{(j_1,\ldots,j_k)\in\JJ_{m,k}\nonumber\\ 
l\not\in\{j_1,\ldots,j_k\}}} 
(1-\alpha_{l})\beta_{m,l;\{j_1,\ldots,j_k\}}
\sigma_{m;\{j_1,\ldots,j_k\}}\nonumber\\
&=\beta_{m,l;\emp} \big((1-\alpha_{l}) \|W_m\circ\cdots\circ
W_1\|+\alpha_{l}\|W_m\circ \cdots W_{l+1}\|\, 
\|W_{l}\circ\cdots\circ W_1\|\big)\nonumber\\
&\quad\;+\sum_{k=1}^{m-2}\sum_{\substack{(j_1,\ldots,j_k)\in
\JJ_{m,k}\\ 
l\not\in \{j_1,\ldots,j_k\}}}
\beta_{m,l;\{j_1,\ldots,j_k\}}
\big((1-\alpha_{l})
\sigma_{m;\{j_1,\ldots,j_k\}}+
\alpha_{l}\sigma_{m;\{j_1,\ldots,j_k\}\cup\{l\}}\big).
\end{align}
In view of \eqref{e:ineqcalW} we conclude that 
\begin{align}
\dfrac{\partial\theta_m}{\partial\alpha_{l}}
(\alpha_1,\ldots,\alpha_{m-1})
&=\beta_{m,l;\emp} \big(\|W_m\circ \cdots W_{l+1}\|\, 
\|W_{l}\circ\cdots 
\circ W_1\|-\|W_m\circ\cdots\circ W_1\|\big)\nonumber\\
&\quad\;+\sum_{k=1}^{m-2}
\sum_{\substack{(j_1,\ldots,j_k)\in\JJ_{m,k}\\ l \not\in
\{j_1,\ldots,j_k\}}}\beta_{m,l;\{j_1,\ldots,j_k\}}
\big(\sigma_{m;\{j_1,\ldots,j_k\}\cup\{l\}}-
\sigma_{m;\{j_1,\ldots,j_k\}}\big)
\geq 0.
\end{align}

\subsection{Proof of Theorem~\ref{t:sep}}
\ref{t:sepi}:
For every $i\in\{1,\ldots,m\}$, set $P_i=R_i(\cdot+b_i)$ and
$(\forall k\in\bK_i)$ 
$\pi_{i,k}=$ $\varrho_{i,k}(\cdot+\scal{b_i}{e_{i,k}})$.
Note that,
for every $i\in\{1,\ldots,m\}$ and every $k\in\bK_i$, $\pi_{i,k}$
is $\alpha_i$-averaged. Furthermore, 
$(\forall i\in\{1,\ldots,m-1\})
(\forall x\in\HH_i)$ $P_ix=\sum_{k\in\bK_i}
\pi_{i,k}(\scal{x}{e_{i,k}})e_{i,k}$.
Now fix $x$ and $y$ in $\HH_0$. It follows from \eqref{e:defTi} 
and the nonexpansiveness of $P_m$ that
\begin{equation}
\|Tx-Ty\|\leq  
\|(W_m\circ P_{m-1}\circ W_{m-1}\circ\cdots\circ P_1\circ W_1)x-
(W_m\circ P_{m-1}\circ W_{m-1}\circ\cdots\circ P_1\circ W_1)y\|.
\end{equation}
In view of Lemma~\ref{l:alphascal}, for every
$i\in\{1,\ldots,m-1\}$, there exists 
$\Lambda_i\in\mathscr{D}_{[1-2\alpha_i,1]}(E_i)$ such that
\begin{multline}
(P_i\circ W_i\circ\cdots\circ P_1\circ W_1)x-
(P_i\circ W_i\circ\cdots\circ P_1\circ W_1)y\\
=\Lambda_i\Big(
(W_i\circ P_{i-1}\circ\cdots\circ P_1\circ W_1)x-
(W_i\circ P_{i-1}\circ\cdots\circ P_1\circ W_1)y\Big).
\end{multline}
Recursive application of this identity yields 
\begin{multline}
(P_{m-1}\circ W_{m-1}\circ\cdots\circ P_1\circ W_1)x-
(P_{m-1}\circ W_{m-1}\circ\cdots\circ P_1\circ W_1)y\\
=(\Lambda_{m-1}\circ W_{m-1}\circ \cdots\circ
\Lambda_1\circ W_1)(x-y).
\end{multline}
This implies that
$\|Tx-Ty\|\leq\|W_m\circ\Lambda_{m-1}\circ\cdots\circ\Lambda_1\circ
W_1\|\,\|x-y\|$. Thus, 
\begin{equation}
\label{e:thetaminterv}
\vartheta_m= 
\sup_{\substack{\Lambda_1\in\mathscr{D}_{[1-2\alpha_1,1]}
(E_1)\\\vdots\\
\Lambda_{m-1}\in\mathscr{D}_{[1-2\alpha_{m-1},1]}
(E_{m-1})}}\|W_m\circ\Lambda_{m-1}\circ\cdots\circ
\Lambda_1\circ W_1\|.
\end{equation}
is a Lipschitz constant of $T$.
Set $S=\{1-2\alpha_1,1\}^{\bK_1}\times\cdots\times
\{1-2\alpha_{m-1},1\}^{\bK_{m-1}}$ and 
$C=[1-2\alpha_1,1]^{\bK_1}\times\cdots\times
[1-2\alpha_{m-1},1]^{\bK_{m-1}}$. 
For every $i\in\{1,\ldots,m-1\}$, 
$\Lambda_i\colon\HH_i\to\HH_i$ is generated from a sequence 
$(\lambda_{i,k})_{k\in\bK_i}$ in $[1-2\alpha_i,1]$ 
via the construction of \eqref{e:defLambda}. The function
\begin{equation}
\begin{array}{rcl}
\psi\colon&C&\to\;\RR\\
&\;\;\big((\lambda_{1,k})_{k\in\bK_1},\ldots,
(\lambda_{m-1,k})_{k\in\bK_{m-1}}\big)&\mapsto\;
\|W_m\circ\Lambda_{m-1}\circ\cdots\circ\Lambda_1\circ W_1\|
\end{array}
\end{equation}
is convex with respect to each of its coordinates. 
Hence, we deduce from Lemma~\ref{l:sup} that
$\sup\psi(C)=\sup\psi(\conv S)=\sup\psi(S)$, 
as claimed.

\ref{t:sepii}:
For every $i\in\{1,\ldots,m-1\}$, the identity operator 
$\Id_i$ of $\HH_i$ lies in
$\mathscr{D}_{\{1-2\alpha_i,1\}}(E_i)$. Hence,
$\vartheta_m\geq\|W_m\circ\Id_{m-1}\circ\cdots\circ\Id_1\circ W_1\|=
\|W_m\circ\cdots\circ W_1\|$.
For every $i\in\{1,\ldots,m-1\}$, let
$\Lambda_i\in\mathscr{D}_{\{1-2\alpha_i,1\}}(E_i)$ and note
that the linear operator
\begin{equation}
\Theta_i= 
\begin{cases}
\dfrac{\Lambda_i-(1-\alpha_i)\Id_i}{\alpha_i},
&\text{if}\;\;\alpha_i\neq 0;\\
0,&\text{otherwise}
\end{cases}
\end{equation}
is nonexpansive. Using the same kind of decomposition as in
the proof of Theorem~\ref{t:propLipgen} yields
\begin{align}
&\hskip -5mm
\big\|W_m\circ\Lambda_{m-1}\circ\cdots\circ \Lambda_1\circ
W_1\big\|\nonumber\\
&=\big\|W_m\circ\big((1-\alpha_{m-1})\Id_{m-1}+\alpha_{m-1}
\Theta_{m-1}\big)\circ\cdots\circ\big((1-\alpha_1)\Id_1+\alpha_1
\Theta_1\big)\circ W_1\big\|
\leq\theta_m\nonumber
\end{align}
and allows us to conclude that $\vartheta_m\leq\theta_m$. 

\subsection{Proof of Proposition~\ref{p:35}}
It follows from \eqref{e:thetaminterv} that
\begin{align}
\vartheta_m(\alpha_1,\ldots,\alpha_{m-1})
&=\sup_{\substack{\Lambda_1\in\mathscr{D}_{[1-2\alpha_1,1]}
(E_1)\\\vdots\\
\Lambda_{m-1}\in\mathscr{D}_{[1-2\alpha_{m-1},1]}
(E_{m-1})}}\|W_m\circ\Lambda_{m-1}\circ\cdots\circ
\Lambda_1\circ W_1\|\nonumber\\
&\leq\sup_{\substack{\Lambda_1\in\mathscr{D}_{[1-2\alpha'_1,1]}
(E_1)\\\vdots\\
\Lambda_{m-1}\in\mathscr{D}_{[1-2\alpha'_{m-1},1]}
(E_{m-1})}}\|W_m\circ\Lambda_{m-1}\circ\cdots\circ
\Lambda_1\circ W_1\|
\nonumber\\
&=\vartheta_m(\alpha'_1,\ldots,\alpha'_{m-1}).
\end{align}

\subsection{Proof of Proposition~\ref{p:sepbis}}
Let us first note that, because of the embeddings, 
$W_1\colon\GG_0\to\HH_1$ is continuous and, likewise, every
$\Lambda_m\in\mathscr{D}_{[1-2\alpha_m,1]}(E_m)$ 
is continuous from $\HH_m$ to $\GG_m$. Hence,
for every $(\Lambda_i)_{1\leq i\leq m}\in
\mathscr{D}_{[1-2\alpha_1,1]}(E_1)\times\cdots\times
\mathscr{D}_{[1-2\alpha_m,1]}(E_{m-1})$,
$\Lambda_m\circ W_m \circ\cdots\circ\Lambda_1\circ W_1\colon
\GG_0\to\GG_m$ is continuous. 
We now follow the same argument as in the proof of 
Theorem~\ref{t:sep}. Let $x$ and $y$ be in $\GG_0$.
For every $i\in\{1,\ldots,m\}$, there exists 
$\Lambda_i\in\mathscr{D}_{[1-2\alpha_i,1]}(E_i)$
such that
$Tx-Ty=(\Lambda_m\circ W_m\circ \Lambda_{m-1}\circ\cdots\circ
\Lambda_1\circ W_1)(x-y)$.
Thus, $\|Tx-Ty\|_{\GG_m}\leq \|\Lambda_m\circ W_m\circ
\Lambda_{m-1}\circ 
\cdots\circ\Lambda_1\circ W_1\|_{\GG_0,\GG_m}\,\|x-y\|_{\GG_0}$,
which leads to \eqref{e:thetambinbis}. 

\subsection{Proof of Corollary~\ref{c:sepbis}}
Since, for every $x\in\HH_m$, 
$(\scal{x}{e_{m,k}})_{k\in\bK_m}\in\ell^2(\bK_m)$,
it follows from H\"older's inequality that $\|\cdot\|_{\GG_m}$ 
in \eqref{e:normlp} 
is well defined and does provide a continuous
embedding of $\HH_m$ in $\GG_m$. As in the proof of
Theorem~\ref{t:sep}, it is enough to take the supremum in 
\eqref{e:dl231} over
$\boldsymbol{D}=\mathscr{D}_{[1-2\alpha_1,1]}(E_1)\times\cdots\times
\mathscr{D}_{[1-2\alpha_{m-1},1]}(E_{m-1})$.
For every $i\in\{1,\ldots,m\}$,
let $\Lambda_i\in\mathscr{D}_{[1-2\alpha_i,1]}(E_i)$. Then
\begin{equation}
\label{e:cyu}
\|\Lambda_m\circ W_m\circ \Lambda_{m-1}\circ\cdots\circ \Lambda_1
\circ W_1\|_{\GG_0,\GG_m}
\leq\|\Lambda_m\|_{\GG_m,\GG_m}\,\| W_m\circ \Lambda_{m-1}\circ 
\cdots\circ\Lambda_1\circ W_1\|_{\GG_0,\GG_m}.
\end{equation}
Let us designate by $(\lambda_{m,k})_{k\in\bK_m}$ the sequence 
in $[1-2\alpha_m,1]$ involved in the construction of
$\Lambda_m$ in \eqref{e:defLambda}.
If $p<\pinf$, then 
\begin{align}
(\forall x\in\HH_m)\quad\|\Lambda_m x\|_{\GG_m} 
&=\Bigg\|\sum_{k\in\bK_m}\lambda_{m,k}\scal{x}{e_{m,k}} 
e_{m,k}\Bigg\|_{\GG_m}
=\Bigg|\sum_{k\in\bK_m}\omega_k|\lambda_{m,k}
\scal{x}{e_{m,k}}|^{p}\Bigg|^{1/p}
\nonumber\\
&\leq\Bigg|\sum_{k\in\bK_m}
\omega_k|\scal{x}{e_{m,k}}|^{p}\Bigg|^{1/p}
=\|x\|_{\GG_m},
\end{align}
which shows that $\|\Lambda_m\|_{\GG_m,\GG_m}\leq 1$.
This inequality holds analogously if $p=\pinf$.
We then deduce from \eqref{e:cyu} that
$\vartheta_m\leq
\sup_{(\Lambda_1,\ldots,\Lambda_{m-1})\in\boldsymbol{D}}
\|W_m\circ\Lambda_{m-1}\circ\cdots\circ\Lambda_1\circ
W_1\|_{\GG_0,\GG_m}$.
On the other hand, it follows from \eqref{e:thetambinbis} that
\begin{equation}
\label{e:minvarthetamlin}
\vartheta_m\geq
\sup_{(\Lambda_1,\ldots,\Lambda_{m-1})\in\boldsymbol{D}} 
\|\Id_m\circ W_m\circ\Lambda_{m-1}\circ\cdots\circ 
\Lambda_1\circ W_1\|_{\GG_0,\GG_m},
\end{equation}
which concludes the proof. 

\subsection{Proof of Proposition~\ref{p:sepbispos}}
For every $i\in\{1,\ldots,m-1\}$, let
$\Lambda_i\in\mathscr{D}_{\{1-2\alpha_i,1\}}(E_i)$ and let
$(\lambda_{i,k})_{k\in\bK_i}$ be the associated sequence in
\eqref{e:defLambda}. Define 
\begin{equation}
\label{e:dl230}
(\forall k\in\bK_m) \quad\lambda_{m,k}=
\begin{cases}
-1,&\text{if}\;\;(\exi (k_0,\ldots,k_{m-1})\in\bK_0\times
\cdots\times \bK_{m-1})\;\mu_{k_0,\ldots,k_{m-1},k}  < 0;\\
1,&\text{otherwise,}
\end{cases}
\end{equation}
and set
$\Lambda_m\colon\HH_m\to\HH_m\colon
x\mapsto\sum_{k\in\bK_m}\lambda_{m,k}\scal{x}{e_{m,k}} e_{m,k}$
and $V_m=\Lambda_mW_m$. Then, by \eqref{e:condpos},
\begin{multline} 
\big(\forall (k_0,\ldots,k_m)\in\bK_0\times\cdots\times
\bK_m\big)\\
\scal{W_1e_{0,k_0}}{e_{1,k_1}} 
\cdots 
\scal{W_{m-1}e_{m-2,k_{m-2}}}{e_{m-1,k_{m-1}}}
\scal{V_me_{m-1,k_{m-1}}}{e_{m,k_m}}
\geq 0.
\end{multline}
In addition, it follows from \eqref{e:normlp} and 
\eqref{e:dl230} that
\begin{align}
\|W_m\circ \Lambda_{m-1}\circ\cdots\circ\Lambda_1\circ
W_1\|_{\GG_0,\GG_m} 
&=\|\Lambda_m\circ V_m\circ \Lambda_{m-1}\circ
W_{m-1}\circ\cdots\circ\Lambda_1
\circ W_1\|_{\GG_0,\GG_m}\nonumber\\
&=\|V_m\circ \Lambda_{m-1}\circ
W_{m-1}\circ\cdots\circ\Lambda_1
\circ W_1\|_{\GG_0,\GG_m}.
\end{align}
Therefore, without loss of generality, we assume that
\begin{equation} 
\label{e:condposbis}
\big(\forall (k_0,\ldots,k_m)\in\bK_0\times\cdots\times
\bK_m\big) \qquad \mu_{k_0,\ldots,k_m}\geq 0.
\end{equation}
Let us now show that
\begin{equation}
\label{e:boundlinnonlin}
\|W_m\circ\Lambda_{m-1}\circ\cdots\circ\Lambda_1\circ
W_1\|_{\GG_0,\GG_m}\leq\|W_m\circ\cdots\circ W_1\|_{\GG_0,\GG_m}.
\end{equation}
Let $\varepsilon\in\RPP$. Then there exists $x\in\HH_0$ such that
$\|x\|_{\GG_0}=1$ and
\begin{equation}
\label{e:boundlinnonlineps}
\|W_m\circ \Lambda_{m-1}\circ\cdots\circ\Lambda_1\circ
W_1\|_{\GG_0,\GG_m}\leq\|(W_m\circ\Lambda_{m-1}
\circ\cdots\circ\Lambda_1
\circ W_1)x\|_{\GG_m}+\varepsilon.
\end{equation}
If $p<+\infty$ in \eqref{e:normlp}, this yields
\begin{multline}
\|W_m\circ\Lambda_{m-1}\circ\cdots\circ\Lambda_1\circ
W_1\|_{\GG_0,\GG_m}\\
\leq\Bigg|\sum_{k_m\in\bK_m}\omega_{k_m} 
\abscal{(W_m\circ \Lambda_{m-1}\circ\cdots\circ\Lambda_1
\circ W_1)x}{e_{m,k_m}}^p\Bigg|^{1/p}+\varepsilon.
\end{multline}
On the other hand, 
\begin{multline}
(W_m\circ\Lambda_{m-1}\circ\cdots\circ
\Lambda_1\circ W_1)x\\
=\sum_{k_{m-1}\in \bK_{m-1}} 
\scal{(\Lambda_{m-1}\circ W_{m-1}\circ\cdots\circ
\Lambda_1\circ W_1)x}{e_{m-1,k_{m-1}}} 
W_me_{m-1,k_{m-1}}
\end{multline}
which, in view of \eqref{e:defLambda}, implies that
\begin{align*}
\hskip -5mm 
(\forall k_m\in\bK_m)\;\;
&\scal{(W_m\circ\Lambda_{m-1}\circ\cdots\circ
\Lambda_1\circ W_1)x}{e_{m,k_m}}\nonumber\\
&\hskip -16mm
=\sum_{k_{m-1}\in \bK_{m-1}} 
\scal{(\Lambda_{m-1}\circ W_{m-1}\circ\cdots\circ
\Lambda_1\circ W_1)x}{e_{m-1,k_{m-1}}} 
\scal{W_me_{m-1,k_{m-1}}}{e_{m,k_m}}\nonumber\\
&\hskip -16mm
=\sum_{k_{m-1}\in \bK_{m-1}} 
\lambda_{m-1,k_{m-1}}\scal{W_me_{m-1,k_{m-1}}}{e_{m,k_m}}
\scal{(W_{m-1}\circ\cdots\circ
\Lambda_1\circ W_1)x}{e_{m-1,k_{m-1}}}.
\end{align*}
Using \eqref{e:elt} recursively yields
\begin{align}
(\forall k_m\in\bK_m)\quad
&\scal{(W_m\circ\Lambda_{m-1}\circ\cdots\circ
\Lambda_1\circ W_1)x}{e_{m,k_m}}\nonumber\\
=&
\sum_{(k_0,\ldots,k_{m-1})\in\bK_0\times\dots\times
\bK_{m-1}}\mu_{k_0,\ldots,k_m}\lambda_{m-1,k_{m-1}}
\cdots\lambda_{1,k_1}\scal{x}{e_{0,k_0}}.
\end{align}
We then deduce from \eqref{e:condposbis} that
\begin{align}
\label{e:boundscalabse0x}
&\hskip -3mm
(\forall k_m\in\bK_m)\quad
\big|\scal{(W_m\circ\Lambda_{m-1}\circ\cdots\circ
\Lambda_1\circ W_1)x}{e_{m,k_m}}\big|\nonumber\\
&\hskip 26mm=
\Bigg|\sum_{(k_0,\ldots,k_{m-1})\in\bK_0\times\dots\times
\bK_{m-1}}\mu_{k_0,\ldots,k_m}\lambda_{m-1,k_{m-1}}
\cdots\lambda_{1,k_1}\scal{x}{e_{0,k_0}}\Bigg|\nonumber\\
&\hskip 26mm\leq\sum_{(k_0,\ldots,k_{m-1})\in 
\bK_0\times\dots\times\bK_{m-1}}\mu_{k_0,\ldots,k_m}
|\lambda_{m-1,k_{m-1}}|\cdots|\lambda_{1,k_1}|\, 
\abscal{x}{e_{0,k_0}}\nonumber\\
&\hskip 26mm\leq\sum_{(k_0,\ldots,k_{m-1})\in\bK_0\times \dots \times
\bK_{m-1}}\mu_{k_0,\ldots,k_m}\abscal{x}{e_{0,k_0}}.
\end{align}
Set $y=\sum_{k_0\in\bK_0}\abscal{x}{e_{0,k_0}}e_{0,k_0}$.
In view of \eqref{e:symnorm0}, $\|y\|_{\GG_0}=\|x\|_{\GG_0}=1$.
Thus, \eqref{e:boundscalabse0x} yields
\begin{align}
\label{e:boundscalabse0xbis}
\big|\scal{(W_m\circ\Lambda_{m-1}\circ\cdots\circ\Lambda_1
\circ W_1)x}{e_{m,k_m}}\big| 
&\leq \sum_{(k_0,\ldots,k_{m-1})\in\bK_0\times \dots \times
\bK_{m-1}}\mu_{k_0,\ldots,k_m}\scal{y}{e_{0,k_0}}\nonumber\\
&=\scal{(W_m\circ\cdots\circ W_1)y}{e_{m,k_m}}.
\end{align}
It then follows from \eqref{e:boundlinnonlineps} and 
the fact that $\|y\|_{\GG_0}=1$ that
\begin{align}
\|W_m\circ\Lambda_{m-1}\circ\cdots\circ\Lambda_1\circ
W_1\|_{\GG_0,\GG_m}
&\leq\Bigg|\sum_{k_m\in\bK_m}\omega_{k_m} 
|\scal{(W_m\circ\cdots\circ W_1)y}{e_{m,k_m}}|^p\Bigg|^{1/p}
+\varepsilon\nonumber\\
&\leq\|(W_m\circ\cdots\circ W_1)y\|_{\GG_m}+\varepsilon\nonumber\\
&\leq\|W_m\circ\cdots\circ W_1\|_{\GG_0,\GG_m}+\varepsilon.
\end{align}
The same inequality is obtained similarly for $p=+\infty$.
This establishes \eqref{e:boundlinnonlin}, which leads to
\begin{equation}
\sup_{\substack{\Lambda_1\in 
\mathscr{D}_{\{1-2\alpha_1,1\}}(E_1)\\
\vdots\\
\Lambda_{m-1}\in\mathscr{D}_{\{1-2\alpha_{m-1},1\}}(E_{m-1})}} 
\|W_m\circ \Lambda_{m-1}\circ\cdots\circ\Lambda_1
\circ W_1\|_{\GG_0,\GG_m}
\leq\|W_m\circ\cdots\circ W_1\|_{\GG_0,\GG_m}.
\end{equation}
Since the converse inequality holds straightforwardly, the proof 
is complete.

\subsection{Proof of Proposition~\ref{p:sepbisposb}}
We use arguments similar to those of the proof of 
Proposition~\ref{p:sepbispos}. For every $i\in\{1,\ldots,m-1\}$, 
let $\Lambda_i\in\mathscr{D}_{\{1-2\alpha_i,1\}}(E_i)$. 
There exists $x\in\HH_0$ such that $\|x\|_{\GG_0}=1$ and
\begin{equation}
\label{e:boundlinnonlinepsb}
\|W_m\Lambda_{m-1}\cdots\Lambda_1 W_1\|_{\GG_0,\GG_m}=
\|(W_m\Lambda_{m-1}\cdots\Lambda_1 W_1)x\|_{\GG_m}.
\end{equation}
On the other hand, for every $k_m\in \bK_m$,
\begin{equation}
\label{e:boundscalabse0xb}
\big|\scal{W_m\Lambda_{m-1}\cdots\Lambda_1 W_1x}{e_{m,k_m}}\big|
\leq\sum_{(k_0,\ldots,k_{m-1})\in\bK_0\times \dots \times
\bK_{m-1}}|\mu_{k_0,\ldots,k_m}|\abscal{x}{e_{0,k_0}}.
\end{equation}
Setting $y=\sum_{k_0\in\bK_0}\abscal{x}{e_{0,k_0}}e_{0,k_0}$ yields
$\big|\scal{W_m\Lambda_{m-1}\cdots\Lambda_1 W_1x}
{e_{m,k_m}}\big|$\linebreak
$\leq\scal{(A_m\cdots A_1)y}{e_{m,k_m}}$,
and \eqref{e:boundlinnonlinepsb} implies that
$\|W_m\Lambda_{m-1}\cdots\Lambda_1 W_1\|_{\GG_0,\GG_m}$\linebreak
$\leq\|A_m\cdots A_1y\|_{\GG_m}\leq\|A_m\cdots A_1\|_{\GG_0,\GG_m}$,
which concludes the proof.

\end{document}